\documentclass[11pt, reqno]{amsart}
\usepackage{amssymb, amscd, mathrsfs}% wasysym} 
\usepackage{a4wide}
\usepackage{color,ulem,textcomp}

\numberwithin{equation}{section}
\theoremstyle{plain}
\newtheorem{theorem}{Theorem}[section]
\newtheorem{remark}[theorem]{Remark}
\newtheorem{proposition}[theorem]{Proposition}
\newtheorem{corollary}[theorem]{Corollary}
\newtheorem{lemma}[theorem]{Lemma}

\def\al{\alpha}

\def\ve{\varepsilon}

\def\ph{\varphi} % not \phi
\def\vp{\varphi}
\newcommand{\R}{\mathbb{R}}

\newcommand{\Z}{\mathbb{Z}}

\newcommand{\I}{\infty}

\newcommand{\pd}{\partial}

\newcommand{\lec}{{\ \lesssim \ }}
\newcommand{\gec}{{\ \gtrsim \ }}

\newcommand{\boxend}{\mbox{}\hfill $\Box$}

\renewcommand{\[}{\begin{equation*}}
\renewcommand{\]}{\end{equation*}}
\begin{document}

\title{On uniqueness for the harmonic map heat flow in supercritical dimensions}

\author{Pierre Germain}
\address[P. Germain]{Courant Institute of Mathematical Sciences\\ New York University\\ 251 Mercer Street\\ New York, N.Y. 10012-1185\\ USA}
\email{pgermain@cims.nyu.edu}

\author{Tej-eddine Ghoul} 
\address[T. Ghoul]{Departement of Mathematics, New York University Abu Dhabi, Computational Research building A2
NYUAD, Saadiyat Island
PO Box 129188, Abu Dhabi, United Arab Emirates.}
\email{teg6@nyu.edu}

\author{Hideyuki Miura}
\address[H. Miura]{Department of Mathematical and Computing Sciences\\
 Tokyo Institute of Technology\\
2-12-1\\ O-okayama\\ Meguro, Tokyo 152-8552\\ Japan}
\email{miura@is.titech.ac.jp }

\begin{abstract}
We examine the question of uniqueness for the equivariant reduction of the
harmonic map heat flow in the energy supercritical dimension $d \ge 3$. 
It is shown that, generically, singular data can give rise to two
distinct solutions which are both stable, and satisfy the local energy inequality. 
We also discuss how uniqueness can be retrieved.
\end{abstract}

\subjclass[2000]{35K58, 53C44, 35AC2}
\keywords{harmonic map heat flow, self-similar solutions, uniqueness}
\thanks{\noindent P. G. is partially supported by NSF grant DMS-1501019, a start-up grant from the Courant Institute, and a Sloan fellowship.
H.M. is partially supported by the JSPS grant  25707005.}

\maketitle

\tableofcontents

\section{Introduction}

\subsection{The equation, its scaling and energy}

The harmonic map heat flow is the following evolution equation
$$
\left\{
\begin{array}{l}
\partial_t u - \Delta u = A(u)(\nabla u, \nabla u) \\
u(t=0) = u_0.
\end{array}
\right.
$$
where $u_0 = u_0(x)$ is a map from $\mathbb{R}^d$ to a Riemannian manifold $M \subset \mathbb{R}^k$ with second fundamental form $A$; and the solution $u = u(t,x)$ is a map from $[0,\infty) \times \mathbb{R}^d$ to $M$.

\medskip

The set of solutions is invariant by the scaling transform $u(t,x) \mapsto u(\lambda^2 t,\lambda x)$ for $\lambda>0$. 
This makes scale-invariant spaces such as $L^\infty$ of particular relevance for the data - we will come back to them.

\medskip

The harmonic map heat flow is the gradient flow for the Dirichlet energy $\int |\nabla u|^2$, and as such, it enjoys (formally) an energy inequality. 
It can be localized to yield (formally) the following local energy inequality, 
for a solution of~\eqref{1} on $[0,T] \times \mathbb{R}^d$, with data $u_0$, locally in $H^1_{t,x}$: 
for any $\tau$, $(\psi^\ell)_{\ell = 1\dots d}$ in $\mathcal{C}^\infty_0 ([0,\infty) \times \mathbb{R}^d)$,
\begin{equation}
\label{EnIn}
\begin{split}
& \int_{\mathbb{R}^d} \frac{1}{2} \tau(t=0) |\nabla u_0|^2 \,dx \\
& \geq \int_0^\infty \int_{\mathbb{R}^d} \left[ \tau |\partial_t u|^2 - \frac{1}{2} \left( \partial_t \tau + \partial_\ell \psi^\ell \right) |\nabla u|^2 + \partial_i \tau\partial_i u^k  \partial_t u^k + \psi^\ell \partial_t u^k  \partial_\ell u^k + \partial_i \psi^\ell \partial_i u^k  \partial_\ell u^k\right] \,dx\,dt.
\end{split}
\end{equation}

Comparing the scaling of the energy with that of the equation gives rise to the notion of criticality.
In the following, we will focus on the supercritical case $\mathbf{d\geq 3}$.

\subsection{The solutions of the harmonic map heat flow and their uniqueness}

The harmonic map heat flow was introduced by Eells and Sampson~\cite{ES}, who also showed that the solution is globally smooth for smooth data if the sectional curvature of $M$ is negative. 
Without this assumption, global smooth solutions are also guaranteed if the image of $u_0$ is contained in a ball of radius $\frac{\pi}{2 \sqrt \kappa}$, 
where $\kappa$ is an upper bound for the sectional curvature of $M$ (see Jost~\cite{Jost}, Lin-Wang~\cite{LW}). Notice that in the case where $M$ is a sphere, 
this corresponds to $u_0$ being valued in a hemisphere, a condition which will also play a role in the present article. 
It was then discovered that, without these assumptions, singularities can form out of smooth solutions: see Coron-Ghidaglia~\cite{CG} and Chen-Ding~\cite{CD2}.

Weak solutions can be built up for less regular data, simply of finite Dirichlet energy. This was established by Chen~\cite{Chen} and Rubinstein-Sternberg-Keller~\cite{RSK}
for the case where the target is a sphere, and extended by Chen-Struwe~\cite{CS} 
to general manifolds, with the help of the monotonicity formula discovered by Struwe~\cite{Struwe}.

Under which conditions are these weak solutions unique? Which is the largest space for the data yielding unique solutions?
Successive improvements gave $W^{1,d}$ (Lin-Wang~\cite{LW2}), 
and, under a smallness condition on the data and the solution, $L^\infty$ (Koch-Lamm~\cite{KL}), 
and $BMO$ (Wang~\cite{Wang}).

These last results are optimal: indeed, for data which are large in $L^\infty$, but not continuous, 
uniqueness is lost, and examples of non-unique solutions can be constructed, 
see Coron~\cite{Coron}, and Bethuel-Coron-Ghidaglia-Soyeur~\cite{BCGS}. 
For these examples however, uniqueness can be salvaged if one requires that solutions satisfy the monotonicity formula 
(itself essentially a consequence of the local energy inequality, see Moser~\cite{Moser}). 
This led Struwe~\cite{Struwe2} to ask whether this criterion could indeed imply uniqueness. 
Germain and Rupflin~\cite{GR} showed that this could not be the case, at least for very specific data, having in particular infinite energy.

These developments lead to the following questions;  {\bf can one describe the loss of uniqueness (or well-posedness) for generic large initial data in $L^\infty$? Can one find a criterion to restore uniqueness in a 
meaningful way?}

In the present paper, we try and answer this question in the framework of equivariant solutions, which will be presented in the following subsection.

\subsection{The equivariant reduction and self-similar solutions}

For the rest of this paper, the target manifold $M$ is the $d$-sphere $\mathbb{S}^d \subset \mathbb{R}^{d+1}$. The harmonic map heat flow equation becomes
\begin{align}
\begin{cases}
\partial_t u -\Delta u=|\nabla u|^2 u
\\
u(t=0)=u_0.
\end{cases}
\label{0}
\end{align} 
The corotational ansatz which we adopt requires that
\begin{equation}
\label{corotationalansatz}
u(t,x) = \left( 
\begin{array}{l} \cos(h(t,|x|)) \\ \sin(h(t,|x|)) \frac{x}{|x|} \end{array}
\right)
\end{equation}
for a radial function $h = h(t,r)$. Under this ansatz, the problem reduces to a scalar, radially symmetric PDE:
\begin{align}
\begin{cases}
\displaystyle{h_t -h_{rr}-\frac{d-1}r h_r+\frac{d-1}{2r^2} \sin (2h)=0} \\
h(t=0)=h_0.
\end{cases}
\label{1}
\end{align}
Notice right away that $h\equiv 0$, $h \equiv \frac{\pi}{2}$, and $h \equiv \pi$ are stationary solutions of this equation; these particular solutions will be important in the following, and correspond geometrically to the north pole, the equator, and the south pole of the sphere.

For $d \geq 3$ (the supercritical case, which will occupy us), self-similar solutions play a key role in the dynamics. 
In particular, for $3\le d \le 6$, it was observed numerically by Biernat and Bizon~\cite{BB} that typical solutions develop singularities at the origin in space, which are then smoothly "resolved" after the singular time. 
Both of these phenomena are asymptotically self-similar: when a singularity is formed at time $t_0$, 
it can be described asymptotically by a "shrinker solution" $\Psi \left( \frac{r}{\sqrt{t_0 - t}} \right)$, 
while the resolution of the singularity is given by an "expander solution" $\psi \left( \frac{r}{\sqrt{t-t_0}} \right)$.

Shrinkers and expanders were 
studied analytically by Fan~\cite{fan} and Germain-Rupflin~\cite{GR} respectively. 
An expander solution $h(t,r) = \psi \left( \frac{r}{\sqrt{t}} \right)$ corresponds to very particular Cauchy data, namely  $h_0 \equiv \psi(\infty)$, the limit of $\psi$ at $\infty$. 
Germain-Rupflin observed that different expanders share the same limit $\psi(\infty)$, yielding an example of non-uniqueness for the harmonic map heat flow even if the local energy inequality is required.

One of the goals of this article is showing the nonlinear stability of these self-similar expanders, 
and thus the genericity of the non-uniqueness result. 

\bigskip \bigskip

\subsection{Obtained results}

\subsubsection{Expanders}

The starting point of our investigations is expander solutions 
of~\eqref{1}, hence of the form 
$h(t,r) = \psi \left( \frac{r}{\sqrt t} \right)$ with constant data $h(t=0) = \psi(\infty) \in [0,\pi]$. 
A direct computation shows that they solve the ODE
\begin{equation}
\label{ODEpsi}
\psi ''+\left(\frac{d-1}\rho+\frac{\rho}2\right)\psi'-\frac{d-1}{2\rho^2} \sin (2\psi)=0.
\end{equation}
For expanders to be smooth and non-constant, we need to impose boundary data which are either of the form
\begin{equation}
\label{northpole}
\begin{cases}
\displaystyle \psi(0)=0 \\
\displaystyle \psi'(0)=\alpha,
\end{cases}
\end{equation}
(for $\alpha \in \mathbb{R}$) or of the form
\begin{equation}
\label{southpole}
\begin{cases}
\displaystyle \psi(0)=\pi \\
\displaystyle \psi'(0)=\alpha
\end{cases}
\end{equation}
(for $\alpha \in \mathbb{R}$). The former case, $\psi(0)=0$, corresponds to $\psi(0)$ on the north pole of the sphere, while the latter case $\psi(0)=\pi$, corresponds to $\psi(0)$ on the south pole of the sphere. We therefore refer to them as north pole, or south pole, boundary conditions, respectively.
Note that the equation \eqref{ODEpsi} is obviously invariant under 
the transform: $\psi$ $\mapsto$ $\pi-\psi$.

The following theorem combines the existence result 
from Germain-Rupflin~\cite{GR} with 
new contribution on uniqueness part in (i) and (iii).
\begin{theorem}
\label{thm:profile}
Let $\ell \in \left[ 0,\frac{\pi}{2} \right]$ (the situation being obviously symmetrical if $\ell \in \left[ \frac{\pi}{2} , \pi \right]$).
\begin{itemize}
\item[(i)] If $3 \leq d \leq 6$, there exists a unique profile $\psi = \psi_N[\ell]$ satisfying~\eqref{ODEpsi}, with north pole boundary conditions~\eqref{northpole}, $\psi(\infty) = \ell$, and $\psi([0,\infty)) \subset \left[0,\frac{\pi}{2}\right]$.
\item[(ii)] If $3 \leq d \leq 6$, there exists a constant $\delta^* > 0$ such that: for $\ell \in {[\frac{\pi}{2}-\delta^*,\frac{\pi}{2})}$,
there exists a profile $\psi = \psi_S[\ell]$ satisfying~\eqref{ODEpsi}, with south pole boundary conditions~\eqref{southpole}, $\psi(\infty) = \ell$, and which crosses exactly once $\frac{\pi}{2}$.
\item[(iii)] If $d \geq 7$ and for $\ell \in [0,\pi/2)$, there exists exactly one profile $\psi = \psi[\ell]$ satisfying~\eqref{ODEpsi}, and $\psi(\infty) = \ell$. This profile has north pole boundary conditions~\eqref{northpole}.
\end{itemize}
\end{theorem}

This theorem will be proved in Section~\ref{oriole}, see in particular Corollary~\ref{egret}.

\begin{remark}
Regarding the existence of expanders as well as their stability properties, the different behavior in the cases $d \leq 6$ and $d \geq 7$ 
is linked to the minimizing properties of the equator map $x \mapsto \frac{x}{|x|}$: it is a local minimum of the Dirichlet energy if and only if $d\geq 7$.

More general equivariant setups can be considered, and we believe that results similar to the case considered here can be obtained.
In particular, the minimizing properties of the equator map should be decisive for the question of uniqueness.
\end{remark}

\subsubsection{Stability of expanders} Our next result shows 
that all expanders constructed above are stable at least locally.
In particular, even if  the uniqueness of the expanders is lost for $3\le d \le 6$, 
the expanders are shown to be stable.
\begin{theorem}
\label{theostable} Let
$\psi_N[\ell]$ and $\psi_S[\ell]$ be the expanders given in 
Theorem \ref{thm:profile}.
Consider data $h_0 \in L^\infty $ and such that
$$
h_0(0) = \ell, \quad |h_0(r) - \ell| \lesssim r, \quad |\partial_r h_0(r)| \lesssim \frac{1}{r}, \quad |\partial_r^2 h_0(r)| \lesssim \frac{1}{r^2}.
$$
\begin{itemize}
\item[(i)] If $3 \leq d \leq 6$ and  $\ell \in [0,\pi/2]$, there exist $T>0$ and a solution $h_N \in L^\infty(0,T;L^\infty)$ of~\eqref{1} close to $\psi_N[\ell]$ in the following sense:  for all $\epsilon>0$, there exists $\zeta>0$ such that
$$
t + r < \zeta \quad \implies \quad \left| h_N(t,r) - \psi_N[\ell]\left( \frac{r}{\sqrt t} \right) \right| < \epsilon.
$$
\item[(ii)] If 
$3 \leq d \leq 6$ and $\ell \in \left[ \frac{\pi}{2} - \delta^*,\frac{\pi}{2} \right]$\footnote{$\delta^*$ is the constant appearing
 in Theorem \ref{thm:profile}.}, there exist $T>0$ and a solution $h_S \in L^\infty(0,T;L^\infty)$ of~\eqref{1} close to $\psi_S[\ell]$ in the following sense:  for all $\epsilon>0$, there exists $\zeta>0$ such that
$$
t + r < \zeta \quad \implies \quad \left| h_S(t,r) - \psi_S[\ell]\left( \frac{r}{\sqrt t} \right) \right| < \epsilon.
$$
\item[(iii)] If $d \geq 7$ and $\ell \in [0,\pi/2)$, there exist $T>0$ and a solution $h \in L^\infty(0,T;L^\infty)$ of~\eqref{1} close to $\psi[\ell]$ in the following sense:  for all $\epsilon>0$, there exists $\zeta>0$ such that
$$
t + r < \zeta \quad \implies \quad \left| h(t,r) - \psi[\ell]\left( \frac{r}{\sqrt t} \right) \right| < \epsilon.
$$ 
\end{itemize}
Furthermore, all of these solutions satisfy the local energy inequality.
\end{theorem}

This theorem is proved, along with more precise results, in Section~\ref{oriole2}.

\begin{remark}
The class of singular data\footnote{
Notice that if $h_0(0)\notin \pi \Z$, the map $u(0,r)$ 
defined by \eqref{corotationalansatz} cannot be continuous at $r=0$.} which we consider barely miss the framework for local well-posedness~\cite{KL,LW2,Wang}. As is well-known, proving solvability for large initial data in a scale-invariant space which does not contain $C_0^\infty$ as a dense subset is a delicate question, which cannot be answered by the usual fixed-point argument. 
%See also subsection \ref{sect:related} below.
\end{remark}

\begin{remark}
An interesting point of view is to regard the data $h_0$ as the trace at blow up time of a solution undergoing a self-similar blow up, for instance if the shrinkers of Fan~\cite{fan} for $3\le d \le 6$ can be proved to be stable, which seems very likely. Therefore, the problem we consider contains the problem of continuation after blow up.

For $d \ge 7$, Bizon and Wasserman \cite{BW} showed that there are no self-similar shrinkers, and type II blow up solutions were constructed 
by Biernat and Seki \cite{BS} very recently. If we denote $h_0$ the trace of these blowing up solutions at the blow up time, they satisfy $h_0(0) = \frac{\pi}{2}$, which is not covered by our results.
\end{remark}

\subsubsection{Uniqueness}

The above theorem shows that, for singular data, non-uniqueness occurs as soon as $3 \leq d \leq 6$ and $h_0(0) \in \left[ \frac{\pi}{2} - \delta^*, \frac{\pi}{2} \right]$.
This raises the question of uniqueness: when does unconditional, or conditional uniqueness hold?

\begin{theorem}
\label{theounique}
\begin{itemize}
\item[(i)] If $3 \leq d \leq 6$, then, for any continuous $h_0$ 
with 
$$
0 \leq h_0(0) < \frac{\pi}{2}, \quad |\partial_r h_0 (r)| \lesssim \frac{1}{r}, \quad |\partial_r^2 h_0 (r)| \lesssim \frac{1}{r^2}, 
$$
there exists $T>0$ such that for any sufficiently small $\delta>0$, there is at most one solution to~\eqref{1} in $L^\infty (0,T; L^\infty)$ satisfying the local energy inequality and
\begin{equation}
\label{est:uniq}
0 \leq h(t,r) < \frac{\pi}{2}-\delta \quad \mbox{for $t + r < \delta$.}
\end{equation}
\item[(ii)] If $d \geq 7$, for any $h_0 \in L^\infty$ and $T>0$, 
there is at most one solution $h \in L^\infty (0,T;L^\infty)$. 
\end{itemize}
\end{theorem}
\noindent

This theorem is proved in Section~\ref{sectionuniqueness}.

\begin{remark} 
Note that time continuity in $L^\infty$ is not required for uniqueness; it is important
 because the solutions we are interested in are not continuous at $t=0$ in general. 

We did not try to give optimal conditions for the initial data, in order to have a concise statement. For instance, it suffices in $(i)$ that $h_0$ is continuous at the origin, and that $|\partial_r h_0 (r)| \lesssim \frac{1}{r}$ and $|\partial_r^2 h_0 (r)| \lesssim \frac{1}{r^2}$ hold in a neighbourhood of the origin.
\end{remark}

\subsection{Summary of obtained results and link with Ginzburg-Landau}

As a conclusion, we find that, if $d \geq 7$, uniqueness holds in the largest class possible, $L^\infty_t L^\infty_x$.

If $3 \leq d \leq 6$, the situation is much more intricate:
\begin{itemize}
\item For generic data (say $h$ smooth, with $h(0)$ close to $\frac{\pi}{2}$), two distinct solutions can be built up in $L^\infty_t L^\infty_x$ (Theorem~\ref{theostable}).
\item These two distinct solutions satisfy the local energy inequality, and are stable; thus both seem to be physically relevant.
\item A means of selecting one of the two is given by Theorem~\ref{theounique}: requiring that a ``local maximum principle'' holds: namely, ask that, for $t$ and $r$ small, $h(t,r)-\frac{\pi}{2}$ has the same sign as $h_0(0) - \frac{\pi}{2}$.
\end{itemize}

This last criterion can be reformulated in terms of the Ginzburg-Landau regularization: the penalization problem
$$
\partial_t u^\epsilon - \Delta u^\epsilon + \frac{1}{\varepsilon^2} (|u^\epsilon|^2 - 1) u^\epsilon = 0
$$
converges, as $\epsilon \to 0$, to a solution $u$ of 
the harmonic map heat flow~\eqref{0}. Furthermore, in dimension $d \geq 5$, this solution agrees with that given by the "local maximum principle", see Section~\ref{sectionGL}.

This selection principle is of course very reminiscent of the situation for scalar conservation laws: uniqueness holds for entropy solutions, which can be constructed by viscous regularization, as was showed by Kru\v{z}kov~\cite{Kru}.

\subsection{Related results and questions}
\label{sect:related}
\subsubsection{The nonlinear heat equation}
The nonlinear heat equation
\begin{align}
\begin{cases}
\pd_t u -\Delta u= u^p
\\
u(t=0)=u_0,
\end{cases}
\label{NLH}
\end{align} 
where $u = u(t,x) \geq 0$, and $(t,x) \in [0,\infty) \times \mathbb{R}^d$, shares many features with the equivariant harmonic map heat flow.

Two important critical exponents appear when considering the dynamics of this PDE:
$$
p_S = \frac{d+2}{d-2} \qquad \mbox{and} \qquad p_{JL} = 1 + \frac{4}{d-4-2\sqrt{d-1}}
$$
(the above definitions of $p_S$ and $p_{JL}$ only make sense for $d \geq 3$ and $d \geq 11$ respectively; outside of these ranges, we set $p_S$ and $p_{JL}$ equal to $\infty$). For $p>p_S$, a singular stationary solution is given by
$$
U(x) = \frac{\beta}{|x|^{\frac{2}{p-1}}}, \qquad \mbox{with $\beta = \left( \frac{2}{p-1} \left( d - 2 - \frac{2}{d} \right) \right)^{\frac{1}{p-1}}$}.
$$
An important step in the study of the nonlinear heat equation and its uniqueness properties is due to Galaktionov and Vazquez~\cite{GV}. 
The large forward self-similar solutions and their non-uniqueness are studied in 
Haraux-Weissler~\cite{HW}, Souplet-Weissler~\cite{SW} and Naito~\cite{Naito2, Naito}.
It follows from these works that
\begin{itemize}
\item In the range $p_S < p< p_{JL}$, the stationary solution $U$ is unstable, while general solutions (with singular data) of the above PDE are non-unique.
\item In the range $p > p_{JL}$, the stationary solution $U$ is stable, while general solutions are unique.
\end{itemize}
This is very reminiscent of the properties of the harmonic map heat flow established in the present paper, and the analogy can be summarized in the following way:
\begin{center}
\begin{tabular}{ | c | c | c | }
\hline
Equation & nonlinear heat equation & harmonic map heat flow \\ \hline \hline
Singular stationary solution & $U(x)$ & Equator map $h \equiv \frac{\pi}{2}$ \\ \hline 
\begin{tabular}{cc} Stability of sing. stat. sol. \\ and non-uniqueness of solutions \end{tabular} & $p_S < p < p_{JL}$ & $3 \leq d \leq 6$ \\ \hline
\begin{tabular}{cc} Instability of sing. stat. sol. \\ and uniqueness of solutions \end{tabular} & $p > p_{JL}$ & $d \geq 7$ \\ \hline
\end{tabular}
\end{center}

However, due to the lack of monotonicity of the nonlinear term,  
the analysis of the harmonic map heat flow  
tends to be more difficult than that of the nonlinear heat equation.
 
\subsubsection{The Navier-Stokes equation}

The theory of the Navier-Stokes equation is also 
in many respects parallel to that of the harmonic map heat flow, 
but it scales differently: 
self-similar data are invariant by the transformation $u_0 \mapsto \lambda u_0(\lambda x,\lambda^2 t)$, and $L^d$ is a scale-invariant space for the data.

Similar to the harmonic map heat flow, the Navier-Stokes equation is well-posed for small self-similar data, see for instance Giga-Miyakawa~\cite{GM}, Kozono-Yamazaki~\cite{KY}, Cannone-Planchon~\cite{CP} and Koch-Tataru~\cite{KT}.
For large self-similar data, existence of a smooth self-similar 
solution has been recently shown by 
Jia-Sverak~\cite{JS1} by applying  the Leray-Schauder theorem 
(see also Tsai \cite{Tsai} for discreately self-similar data). 
These authors also suggest in~\cite{JS2} non-uniqueness of the solution for self-similar data as a possible mechanism for ill-posedness of the Navier-Stokes equation in the energy space; such a possibility 
would be very similar to the behavior established here for the 
harmonic map heat flow.

\subsubsection{The Wave Map equation}

The hyperbolic counterpart of the harmonic map heat flow is the wave map;    since it is time reversible, there is no difference between shrinkers and expanders. 
Such self-similar solutions have been constructed by Shatah~\cite{Shatah}, Cazenave-Shatah-Tahvildar-Zadeh~\cite{CSTZ}, 
see also Germain~\cite{PG1,PG2} and Widmayer~\cite{Widmayer} for their uniqueness properties. The nonlinear stability of these solutions (with a perturbation at the blow up time) seems to be an open problem, but stability of blow up 
has been established by Donninger~\cite{Donninger}.

\subsection{Plan of this paper}

The rest of the paper is organized as follows.
In the next section, we prove Theorem \ref{thm:profile} 
and derive qualitative properties for the self-similar expanders 
which will be used in the subsequent sections.
The section 3 and 4 are devoted to the stability of the expanders 
and the uniqueness result (Theorem \ref{theostable} and \ref{theounique}).
In section 5, we study the Ginzburg-Landau regularization equation.
Finally we collect technical results such as 
the comparison principle and the regularity estimates 
for the harmonic map heat flow in the appendix.

\subsection{Notations and conventions}

We use the following conventions:
\begin{itemize}
\item $C$ denotes a constant whose value may change from one line to the next. For quantities $A$ and $B$, we denote $A \lesssim B$ if $A \leq CB$.
\item $\langle y \rangle = \sqrt{1 + |y|^2}$ for $y \in \mathbb{R}$.
\item $x \in \mathbb{R}^d$ usually denotes the space variable and $r=|x|$ is the radial variable.
\end{itemize}

Solutions of PDEs such as~\eqref{0} and~\eqref{1} are always understood in the sense of distributions.
When radial notation is used, such as in~\eqref{1}, this is simply as a notational convenience, and we do not think of the PDE as set on the half line.

For instance, $h \in L^\infty_{t,x}$ solves~\eqref{1} on $[0,T)$ if and only if: for any $\varphi \in \mathcal{C}^\infty_0 ([0,T) \times \mathbb{R}^d)$, 
\begin{align*}
& \int_0^\infty \int_{\mathbb{R}^d} h \left( -\varphi_t- \Delta \varphi \right)\,dx\,dt + \int_0^\infty \int_{\mathbb{R}^d} \frac{d-1}{2|x|^2}
\sin(2h)\varphi \,dx\,dt = \int h_0 \, \varphi(t=0) \,dx.
\end{align*}

Notice that if $u$ is given by the corotational ansatz~\eqref{corotationalansatz}, with $h$ locally in $H^1_{t,x}$ - which includes solutions studied here - the formulations
~\eqref{0} and~\eqref{1} are equivalent. This follows from the equality (in the sense of distributions)
$$
\partial_t u - \Delta u - |\nabla u|^2 u = 
\left( \begin{array}{l} - \sin h \\ (\cos h) \frac{x}{|x|} \end{array} \right) \left[ \partial_t h - \Delta h - \frac{d-1}{2 |x|^2} \sin(2h) \right].
$$

\section {Analysis of self-similar expanders}
\label{oriole}
In this section we derive various properties 
of the self-similar expanders, which play crucial roles 
in the whole of this paper.
 
\subsection{The equation satisfied by self-similar expanders}
Define the self-similar variables
$$
s = \log t, \quad \rho = \frac{r}{\sqrt t},
$$
and let
$$
v(s,\rho) = h(e^s , e^{s/2} \rho).
$$
The equation~\eqref{1} becomes in these variables
\begin{align}
\begin{cases}
\displaystyle{
 v_s - v_{\rho \rho}-\left(\frac{d-1}\rho+\frac{\rho}{2} \right)v_{\rho}+
\frac{d-1}{2\rho^2} 
\sin (2 v)=0,}
\\
v(s,\rho) \sim h_0(e^{s/2} \rho) \quad \mbox{as $s \to -\infty$}.
\end{cases}
\label{2}
\end{align}
Self-similar solutions ("expanders") $\psi(\frac{r}{\sqrt t})$ in the original variables become stationary solutions $\psi(\rho)$ 
in these new variables; namely, they solve
\begin{equation}
\label{eqpsi}
\psi ''+\left(\frac{d-1}\rho+\frac{\rho}2\right)\psi'-\frac{d-1}{2\rho^2} \sin (2\psi)=0.
\end{equation}
First, we notice that reasonable solutions to this ODE must have data $\psi(0) = 0$, $\frac{\pi}{2}$, or $\pi$, and that $\psi(0) = \frac{\pi}{2}$ implies that $\psi \equiv \frac{\pi}{2}$.

Taking into account the symmetry between $0$ and $\pi$ (north and south poles of the sphere), 
we will focus on the following problem for $\psi=\psi_\alpha$:
\begin{align}
\begin{cases}
\displaystyle{\psi ''+\left(\frac{d-1}\rho+\frac{\rho}2\right)\psi'-\frac{d-1}{2\rho^2} \sin (2\psi)=0,} \\
\psi(0)=0, \qquad \psi'(0)=\alpha. 
\end{cases}
\label{ODE}
\end{align}
where $\alpha$ is a given nonnegative constant. The case of south pole boundary conditions~\eqref{southpole} can be immediately deduced by the transform $\psi \mapsto \pi-\psi$.

\subsection{First properties of the profile $\psi_\alpha$}

Part of the results in the following proposition can be found in~\cite{GR}; we then only sketch their proofs and refer to~\cite{GR} for more details. New results are related in particular to the important function
$$
\ph(\rho) = \ph_\alpha(\rho) = \frac{d}{d\alpha} \psi_\alpha(\rho).
$$
The sign of this function controls the stability of expanders since 
it is a zero-eigenfunction of the linearized equation of \eqref{1} around the expanders:  indeed, $\ph_\alpha$ satisfies
\begin{align}
\begin{cases}
\displaystyle{\ph_\al''+\left(\frac{d-1}\rho+\frac{\rho}2\right)\ph_\al '
-\frac{d-1}{\rho^2} \cos (2\psi_\al)\ph_\al=0
\qquad \rm{in} \ \R_+}, \\
\ph_\al(0)=0, \qquad \ph_\al '(0)=1. 
\end{cases}
\label{ODE2}
\end{align}
(we refer to Subsection~\ref{afv} for the precise sense in which the initial data is assumed; and for the construction of the solution in a more general framework).

\begin{proposition}
\label{GR}
\begin{itemize}
\item[(i)] For any $\alpha\geq 0$,~\eqref{ODE} admits a unique global solution. It is such that $\alpha \mapsto \psi_\alpha(\rho)$ is smooth (analytic) for any $\rho$; furthermore $0 \leq \psi_\alpha(\rho) \leq \pi$ for any $\rho$. 
\item[(ii)] The derivative of $\psi$ satisfies the bound $|\psi'(\rho)| \leq \frac{C}{\rho^3}$, locally uniformly in $\alpha$. 
In particular, $\psi(\rho)$ has a limit as $\rho \to \infty$, denoted $\psi(\infty)$.
\item[(iii)] The map $\alpha \mapsto \psi_\alpha (\infty)$ is real analytic.
\item[(iv)] For $\alpha>0$, if $\psi_\alpha([0,R]) \subset [0,\frac{\pi}{2})$, then $\psi'_\alpha(\rho) > 0$ for $\rho \in [0,R]$.
\item[(v)] For $\alpha>0$, if $\psi_\alpha([0,R]) \subset [0,\frac{\pi}{2})$, then $\ph_\alpha(\rho) = \frac{d}{d\alpha} \psi_\alpha(\rho) > 0$ for $\rho \in (0,R]$.
\item[(vi)] If $3 \leq d \leq 6$, for any number $K \in \mathbb{N}$, there exists a neighbourhood $U_K$ of $\frac{\pi}{2}$ such that, if $s \in U_K$, the set $\{ \alpha \; \mbox{such that} \; \psi_\alpha(\infty) = s \}$ has cardinal $\geq K$.
\item[(vii)] If $d \geq7$, then for any $\alpha>0$ and $\rho > 0$, $\psi_\alpha(\rho) < \frac{\pi}{2}$; for any $\alpha > 0$, $\rho \mapsto \psi_\alpha(\rho)$ is increasing; and finally $\lim_{\alpha \to \infty} \psi_\alpha (\rho) = \frac{\pi}{2}$ for any $\rho>0$.
\end{itemize}
\end{proposition}

\begin{proof} $(i)$ In order to prove existence and continuous dependence, the difficulty is to solve in a neighbourhood of zero, but this can be achieved by a fixed point argument, similar to the one used below in the proof of Lemma~\ref{perturbation}.

To prove that $\psi_\alpha$ is valued in $[0,\pi]$, notice that 
$$
\mbox{setting $H(\rho)=\rho^2 \psi'(\rho)^2 - (d-1) (\sin \psi(\rho))^2$,} \qquad \frac{d}{d\rho} H(\rho)= - \rho^2 \psi'(\rho)^2 \left( \rho + \frac{2(d-2)}{\rho} \right) \leq 0.
$$
The conclusion follows since $H(0) = 0$, while $H(\rho)>0$ if $\psi(\rho) = 0$ or $\pi$, and $\psi'(\rho) \neq 0$.

\bigskip

\noindent $(ii)$ To prove this statement, use that the differential inequality above implies that, setting $F(\rho) = \rho^2 \psi'(\rho)^2$, it satisfies
$$
F'(\rho) + \rho F(\rho) \leq (d-1) \sin(2\psi(\rho)) \psi'(\rho) \leq \frac{\rho}{2} F(\rho) + \frac{C}{\rho^3}.
$$
The desired result follows by integrating the differential inequality $F'(\rho) + \frac{\rho}{2} F(\rho) \leq \frac{C}{\rho^3}$; we refer to~\cite{GR} for more details.

\bigskip

\noindent $(iii)$ Since $\psi$ is constructed by a fixed point argument in a neighbourhood of $0$, it follows that $\alpha \mapsto (\psi_\alpha(R),\psi_\alpha'(R))$ is 
analytic for $R$ sufficiently small, and hence for all $R>0$ since the equation is regular away from $\rho=0$. Thus to obtain analyticity of $\alpha \mapsto \psi_\alpha (\infty)$,
it suffices to show that $(\psi_\alpha(R),\psi_\alpha'(R)) \mapsto \psi_\alpha (\infty)$ is analytic, for $R$ sufficiently big. 
To simplify notations, we will only do so for $(\psi_\alpha(R),\psi_\alpha'(R))$ in a neighbourhood of zero.
Observe that solving the equation~\eqref{ODE} for $\rho>R$ with data $(\psi (R), \psi'(R)) = (\beta,\gamma)$ is equivalent to the fixed point problem
$$
\psi(\rho) 
= \beta + \gamma \underbrace{\int_R^\rho s^{1-d} e^{-\frac{s^2}{4}}\,ds}_{\displaystyle \Psi(\rho)} 
+ \underbrace{\frac{d-1}{2} \int_R^\rho s^{1-d} e^{-\frac{s^2}{4}} \int_R^s t^{d-3} e^{\frac{t^2}{4}} \sin(2 \psi(t)) \, dt \, ds}_{\displaystyle T(\psi)(\rho)}.
$$
For $\beta$, $\gamma$ sufficiently small, and $R$ sufficiently large, it is easy to see that the map $\psi \mapsto \beta + \gamma \Psi + T(\psi)$ 
is a contraction in a ball of $L^\infty(R,\infty)$ centered at $0$ of sufficiently small radius. The above problem can then be solved by Neumann series to give
$$
\psi = \sum_{n=0}^\infty T^n (\beta + \gamma \Psi).
$$
Next expand $T$ as
$$
T (\psi) = \sum_{k=0}^\infty \frac{d-1}{2} \frac{(-2)^{2k+1}}{(2k+1)!} \int_R^\rho s^{1-d} e^{-\frac{s^2}{4}} \int_R^s t^{d-3} e^{\frac{t^2}{4}} [\psi(t)]^{2k+1} \, dt \, ds
= \sum_{k=0}^\infty T_k (\psi),
$$
which finally gives the expansion
$$
\psi = \sum_{n=0}^\infty \sum_{k_1, \dots, k_n} T_{k_1} \dots T_{k_n} (\beta + \gamma \Psi),
$$
which converges absolutely in $L^\infty$.
Taking the trace at infinity (which is possible since $T_k f$ converges at $\infty$ for any bounded $f$),
$$
\psi(\infty) = \sum_{n=0}^\infty \sum_{k_1, \dots, k_n} \left[ T_{k_1} \dots T_{k_n} (\beta + \gamma \Psi)\right](\infty).
$$
Expanding a last time in $\beta$ and $\gamma$, this gives the desired convergent expansion, hence analyticity of the map $(\beta , \gamma) \mapsto \psi(\infty)$, 
which was the desired result.

\bigskip

\noindent $(iv)$ Arguing by contradiction, assume that there exists $\rho_0 \in (0,R]$ such that $\psi'_\alpha(\rho_0) = 0$ and $\psi'_\alpha(\rho) > 0$ if $\rho \in [0,\rho_0)$. The equation satisfied by $\psi_\alpha$ implies that $\psi''_\alpha(\rho_0) > 0$, which yields the desired contradiction.

\bigskip

\noindent $(v)$ Define $\underline{\ph}(\rho) = \underline{\ph}_\alpha(\rho) = \alpha^{-1} \rho \psi'_\alpha(\rho)$. It solves
\begin{equation}
\label{ODE1}
\begin{cases}
\displaystyle{\underline{\ph}_\alpha''+\left(\frac{d-1}\rho+\frac{\rho}2\right)
\underline{\ph}_\alpha'-\frac{d-1}{\rho^2} \cos (2\psi_\alpha)\underline{\ph}_\alpha=-\underline{\ph}_\alpha
\qquad \rm{in} \ \R_+,} \\
\underline{\ph}_\alpha(0)=0, \qquad \underline{\ph}'_\alpha(0)=1. 
\end{cases}
\end{equation}
(this is related to time translation invariance of~\eqref{1}). We know from $(iv)$ that if $\psi_\alpha([0,R]) \subset [0,\frac{\pi}{2})$, then $\psi_\alpha'(\rho)>0$ for all $\rho \in [0,R]$. By definition of $\underline{\varphi}_\alpha(\rho)$, this implies that $\underline{\varphi}_\alpha(\rho)>0$ for $\rho \in [0,R]$.

Considering the equations satisfied by $\ph$ and $\underline{\ph}$, the comparison principle (Lemma~\ref{comparison}) implies that $\varphi_\alpha(\rho) \geq \underline{\varphi}_\alpha(\rho)>0$ for all $\rho \in [0,R]$.

\bigskip

\noindent $(vi)$ is proved by comparing the solution of~\eqref{ODE} to that of the ODE for (stationary) harmonic maps
\begin{align}
\begin{cases}
\displaystyle{\zeta ''+ \frac{d-1}r \zeta'-\frac{d-1}{2r^2} \sin (2\zeta)=0
\qquad \rm{in} \ \R_+,} \\
\zeta(0)=0, \qquad \zeta'(0)=1, 
\end{cases}
\end{align}
which was studied in~\cite{JK2}. For $3 \leq d \leq 6$, the solution $\zeta$ converges as $r \to \infty$ to $\frac{\pi}{2}$ by oscillating around this value. Using in addition Corollary~\ref{grebe} gives the result; we refer to~\cite{GR} for more details.

\bigskip

\noindent $(vii)$ We will also use the function $\zeta$ introduced above, using this time that, for $d \geq 7$, it converges to $\frac{\pi}{2}$ and remains increasing for $r>0$ (see~\cite{JK2}). 
Denote $\zeta_{\alpha'}$ the solution of the above equation with data $(\zeta(0), \zeta'(0)) = (0,\alpha')$ 
(simply obtained by dilation from $\zeta$). Using that $\zeta'>0$, it appears that
$$
\zeta ''+\left(\frac{d-1}\rho+\frac{\rho}2\right)\zeta'-\frac{d-1}{2\rho^2} \sin (2\zeta) \geq 0.
$$
Setting $g = \rho^\eta (\zeta_{\alpha'} - \psi_{\alpha})$ (where $\eta$ is a constant whose value will be fixed soon), one checks easily that
$$
g'' + \left( \frac{d-1-2\eta}{\rho} + \frac{\rho}{2} \right) g' + \left( \frac{\eta^2 + (2-d) \eta + d-1}{\rho^2} - \frac{\eta}{2} \right) g \geq 0.
$$
as long as $\zeta_{\alpha'} \geq \psi_{\alpha}$. By choosing $\alpha' = 2\alpha$ sufficiently big, one can ensure that $g'(0)>0$. Therefore, $g$ is increasing in a neighbourhood of zero. 
Furthermore, since $d \geq 7$, it is possible to choose $\eta$ such that the coefficient of $g$ in the above is $<0$. 
But then, $g$ cannot reach a local maximum for $\rho>0$, or the equation on $g$ would be contradicted. Therefore, $g(\rho)>0$ for all $\rho>0$, which implies in particular that $\psi< \frac{\pi}{2}$.

This implies by $(iv)$ that $\rho \mapsto \psi_\alpha(\rho)$ is increasing.

Finally, the fact that $\lim_{\alpha \to \infty} \psi_\alpha (\rho) = \frac{\pi}{2}$ follows by comparison with $\zeta$; the reader is refered to~\cite{GR} for more details.
\end{proof}

\subsection{Study of $\ph_\alpha(\rho) = \frac{d}{d\alpha} \psi_\alpha(\rho)$}

Define
\begin{equation}
\label{defalpha0}
\begin{split}
& \alpha^* = \sup \{ \beta>0 \; \mbox{such that} \; \ph_\alpha(\rho)>0 \; \mbox{for all $\rho>0$ and all $\alpha \in (0,\beta)$} \} \\
& \alpha_0 = \sup \{ \beta>0 \; \mbox{such that} \; \psi_\alpha(\rho) < \frac{\pi}{2} \; \mbox{for all $\rho$ and all $\alpha \in (0,\beta)$}\}
\end{split}
\end{equation}
Notice first that 
$$
\alpha_0, \alpha^* \in (0,\infty].
$$
Indeed, by Proposition~\ref{GR} $(i)$ and $(ii)$, for $\alpha$ small, $\psi_\alpha(\rho)$ is valued in, say  $(0,\frac{1}{100})$ for $\rho>0$. By Proposition~\ref{GR} $(v)$, this implies that, for $\alpha$ small, $\ph_\alpha(\rho)>0$ for $\rho>0$.

\begin{proposition}
\label{monotonicity}
\begin{itemize}
\item[(i)] For any $\alpha$, $\varphi_\alpha(\rho)$ has a limit as $\rho \to \infty$, which we denote $\varphi_\alpha(\infty)$.
\item[(ii)] If $3 \leq d \leq 6$, $\alpha_0 < \alpha^* < \infty$, and in particular $\psi_{\alpha^*}(\infty) > \frac{\pi}{2}$.  
Furthermore, there exists a finite set $\{ \alpha_i,\, i = 1 \dots N \}$, contained in $[\alpha_0,\alpha^*)$, such that $\varphi_\alpha(\infty) > 0$ for all 
$\alpha \in (0,\alpha^*) \setminus \{\alpha_i, \,i = 0\dots N \}$.
\item[(iii)] If $d \geq 7$,  $\alpha_0 = \alpha^* = \infty$, and $\varphi_\alpha(\infty) >0$ for all $\alpha > 0$.
\end{itemize}
\end{proposition}

\noindent
\begin{proof} $(i)$ The convergence of $\varphi_\alpha$ as $\rho \to \infty$ is a consequence of Lemma~\ref{lemmaODE}. 

\bigskip
\noindent
$(iii)$ If $d \geq 7$, Proposition~\ref{GR}  $(vii)$ implies that $\alpha_0 = \infty$, and it is clear from $(v)$ that $\alpha^* \geq \alpha_0$, hence $\alpha^* = \infty$.

Arguing by contradiction, suppose that there exists $\widetilde \alpha \in (0,\infty)$ such that $\vp_{\widetilde \alpha}(\infty)=0$.
By monotonicity of $\alpha \mapsto \psi_\alpha$, $-\cos(2 \psi_{\widetilde \alpha}) < - \cos(2 \psi_\alpha)$ if $\widetilde \alpha < \alpha$.
Then the comparison principle shows that $\varphi_{\widetilde \alpha} \geq \varphi_\alpha \geq 0$ if $\widetilde \alpha < \alpha < \infty$, which implies
$\varphi_\alpha(\infty) = 0$ for the same range of $\alpha$. 

By Proposition~\ref{GR} and Lemma~\ref{lemmaODE}, $\psi_\alpha(\rho)$ and $\ph_\alpha(\rho)$ converge uniformly in $\alpha$, as $\rho \to \infty$, to $\psi_\alpha(\infty)$ and $\ph_\alpha(\infty)$ respectively. Therefore it is possible to commute $\frac{d}{d\alpha}$ and $\lim_{\rho \to \infty}$ to obtain that $\frac{d}{d\alpha} \psi_\alpha (\infty) = \varphi_\alpha(\infty)$.

But then $\varphi_\alpha(\infty) = 0$ for $\widetilde \alpha < \alpha < \infty$ implies that $\psi_\alpha(\infty)\equiv \frac{\pi}{2}$ for $\widetilde \alpha < \alpha < \infty$ by Proposition~\ref{GR} $(vii)$. 
However the behavior of $\psi_{\alpha}(\infty)$ near $\alpha=\tilde{\alpha}$ 
contradicts the analyticity  in Proposition~\ref{GR} $(iii)$.

\bigskip
\noindent
$(ii)$ If $3 \leq d \leq 6$, Proposition~\ref{GR} $(vi)$ implies that $\alpha_0<\infty$ and $\alpha^* < \infty$.

If $\alpha<\alpha_0$, 
%the same arguments as in the case $d \geq 7$ apply, since in this case, 
$\psi_\alpha$ takes values in $(0,\frac{\pi}{2})$, and we learn from Proposition~\ref{GR} $(v)$ that 
$\varphi_\alpha > 0$ 
%This implies that $\varphi_\alpha(\infty)>0$ for $\alpha<\alpha_0$, 
which implies  $\alpha^* \geq \alpha_0$.
We next show $\alpha^* > \alpha_0$. Arguing by contradiction, assume that this is not the case and that $\alpha^* = \alpha_0$. First observe that $\ph_{\alpha_0}(\rho)>0$ for all $\rho>0$ (indeed, $\ph_{\alpha_0}(\rho) \geq0$ for all $\rho$ and if for some $\rho$, $\ph_\alpha(\rho)= \ph_\alpha'(\rho)= 0$, then $\ph_{\alpha} = 0$). Then there are essentially three possibilities:
\begin{itemize}
\item If $\varphi_{\alpha_0}(\infty)>0$, using continuous dependence of the solution on $\alpha$, we can find $\delta = \delta(R) >0$ such that, if $\alpha \in (\alpha_0,\alpha_0+\delta)$, $\varphi_{\alpha}(\infty) > 0$ and $\varphi_{\alpha}(\rho) > 0$ if $1<\rho<R$. By Corollary~\ref{loon}, this implies that $\varphi_{\alpha}(\rho) > 0$ for all $\rho$ if $\alpha \in (\alpha_0,\alpha_0+\delta)$. But this contradicts $\alpha_0 = \alpha^*$.
\item If $\varphi_{\alpha_0}(\infty)=0$, then, relying on analyticity, there exists $\delta > 0$ such that $\varphi_{\alpha}(\infty)$ has a constant, nonzero, sign for $\alpha \in (\alpha_0, \alpha_0 + \delta)$. If $\varphi_{\alpha}(\infty)>0$ for $\alpha \in (\alpha_0, \alpha_0 + \delta)$, one can argue as in the previous case.
\item Therefore matters reduce to the case where $\varphi_{\alpha_0}(\infty)=0$ and $\varphi_{\alpha}(\infty)<0$ for $\alpha \in (\alpha_0, \alpha_0 + \delta)$. This implies that $\psi_\alpha(\infty) < \frac{\pi}{2}$ for $\alpha \in (\alpha_0,\alpha_0 + \delta)$. Taking $\delta$ sufficiently small, we can ensure by continuity of the map $\alpha \mapsto \psi_\alpha(\rho)$ that, for $\alpha \in (\alpha_0,\alpha_0 + \delta)$, $\psi_\alpha(\rho)<\frac{\pi}{2}$ for $\rho<R$, where $R$ can be chosen arbitrarily large. By Corollary~\ref{grebe}, this implies that $\psi_\alpha(\rho)< \frac{\pi}{2}$ for all $\rho>0$ and $\alpha_0 < \alpha < \alpha_0 + \delta$. But then Proposition~\ref{GR} $(v)$ implies that $\ph_\alpha(\infty)\geq 0$ for $\alpha_0 < \alpha < \alpha_0 + \delta$, and thus $\psi_\alpha(\infty)\geq \frac{\pi}{2}$ for $\alpha_0 < \alpha < \alpha_0 + \delta$. 
This is the desired contradiction.
\end{itemize}
Finally for $\alpha_0 < \alpha < \alpha^*$, it is clear that $\varphi_\alpha(\infty)\geq 0$ (indeed, $\ph_\alpha(\rho)>0$ for all $\rho$, so it suffices to take the limit). By analyticity of $\alpha \mapsto \varphi_\alpha(\infty)$, this map can only vanish on a discrete set without accumulation point.

\end{proof}

The following corollary now implies Theorem \ref{thm:profile}.

\begin{corollary}
\label{egret}
\begin{itemize}
\item[(i)]
Given $\ell \in (0,\frac{\pi}{2}]$ if $3 \leq d \leq 6$, and $\ell \in (0,\frac{\pi}{2})$ if $d \geq 7$, there exists a unique $\alpha$ such that 
$$
\psi_\alpha (\infty) = \ell \qquad \mbox{and} \qquad 0 \leq \psi_\alpha(\rho) \leq \frac{\pi}{2} \; \mbox{for all $\rho> 0$}.
$$

\item[(ii)] If $3 \leq d \leq 6$, let
$$
\ell^* = \psi_{\alpha^*}(\infty).
$$
For $\ell \in (\frac{\pi}{2},\ell^*)$, there exists a unique $\alpha \in (\alpha_0,\alpha_*)$ such that $\psi_\alpha(\infty) = \ell$ and $\psi_\alpha$ crosses $\frac{\pi}{2}$ only once.

\item[(iii)] If $d\ge 7$, 
there is no solutions satisfying $\psi_{\alpha}(\infty) \in (0,\frac{\pi}{2})$
and $\psi_{\alpha}(0)=\pi$.

\end{itemize}
\end{corollary}

\begin{proof} $(i)$ If $d \geq 7$, this follows easily from Proposition~\ref{GR}.

It is also an easy consequence of Proposition~\ref{GR} that if $3 \leq d \leq 6$ and $\ell \in (0,\frac{\pi}{2})$, there exists a unique $\alpha \in (0,\alpha_0)$ such that $\psi_\alpha(\infty) = \ell$ (recall that, in this case, $\varphi_\alpha>0$ by $(v)$ in Proposition~\ref{GR}).

If $3 \leq d \leq 6$ and $\alpha > \alpha_0$, let us now show that $\psi_\alpha((0,\infty)) \not\subset [0,\frac{\pi}{2}]$. By definition of $\alpha_0$, there exists $\beta \in (\alpha_0, \alpha)$ and $\rho_0>0$ (which we can assume to be minimal) such that $\psi_{\beta}(\rho_0)=\frac{\pi}{2}$. If, on the one hand, there exists $\rho_1<\rho_0$ such that $\psi_\alpha (\rho_1) = \frac{\pi}{2}$, we are done, since then $\psi_\alpha(\rho) > \frac{\pi}{2}$ for $\rho$ slightly larger than $\rho_1$. If, on the other hand, such a $\rho_1$ does not exist, then by Proposition~\ref{GR} $(v)$, we obtain $\psi_\alpha(\rho_0) = \frac{\pi}{2}$, implying once again that $\psi_\alpha((0,\infty)) \not\subset [0,\frac{\pi}{2}]$.

This follows from Proposition~\ref{GR} $(v)$: once $\psi$ crosses $\frac{\pi}{2}$, it cannot come back.

\bigskip

\noindent $(ii)$ follows from the definition of $\alpha^*$: if $\alpha<\alpha^*$, $\ph_{\alpha}(\rho)>0$ for all $\rho>0$, while $\ph_\alpha(\infty) \geq 0$, and, by analyticity, the map $\alpha \to \psi_\alpha(\infty)$ is a bijection on $(\alpha_0,\alpha^*)$ .

\bigskip 

\noindent $(iii)$ follows from Proposition
~\ref{GR}~$(vii)$. Indeed this yields that for the data 
$\psi_{\alpha}(0)=\pi$, $\psi'_{\alpha}(0)=-\alpha$, we see from the symmetry that 
$\psi_{\alpha}\in (\pi/2,\pi)$ for any $\alpha \ge 0$.
\end{proof}

\subsection{A first variation on $\ph_\alpha$}
\label{afv}
The solution $w$ to the ODE below will be the key to linear and nonlinear stability of expanders, providing a tool similar to a spectral gap result. Define first $Z: [0,\infty) \to [0,\infty)$ to be a smooth increasing function such that
$$
Z(\rho) =
\left\{
\begin{array}{ll}
\rho & \mbox{if $\rho \leq 1$} \\
\rho^2 & \mbox{if $\rho \geq 2$}.
\end{array}
\right.
$$
\begin{lemma}
\label{perturbation}
For 
$$
\alpha \in 
\left\{
\begin{array}{ll}
(0,\alpha_*) \setminus \{ \alpha_i, \, i=1 \dots N \} & \mbox{if $3 \leq d \leq 6$} \\
(0,\infty) & \mbox{if $d \geq 7$}
\end{array}
\right.
$$
there exists $\kappa_0=\kappa_0(\alpha)>0$ such that for each $\kappa <\kappa_0$ the problem
\begin{align}
\begin{cases}
\displaystyle{
w''+\left(\frac{d-1}\rho+\frac{\rho}2\right)w '-\frac{d-1}{\rho^2}
\cos (2\psi_\alpha) w + \frac{\kappa}{Z(\rho)} w =0} \\
\displaystyle{
w(0)=0, \qquad w'(0) = 1, \qquad |w'(\rho) - 1 | \lesssim \rho}
\end{cases}
\label{3}
\end{align}
%\comment{H: the above equation is different from that in the old version.
%It looks easier to deal with. Is it the advantage? P: it improves the existence %result (we can allow $f(r) \lesssim r$)}
has a unique solution $w$, which is furthermore positive for $\rho>0$, and has a positive limit $w(\infty)$ as $\rho \to \infty$.
\end{lemma}

\begin{proof}
\underline{Step 1: Local solvability.}
We focus first on solving the above ODE in a neighborhood of zero. In particular, we can replace $Z(\rho)$ by $\rho$ in that region.
Setting $W(\rho) = \frac{w(\rho)}{\rho}$, the ODE to be solved becomes
\begin{align*}
\begin{cases}
W''+\left(\frac{d+1}\rho+\frac{\rho}2\right)W '- \left[ \frac{d-1}{\rho^2}
\cos (2\psi_\alpha) - \frac{d-1}{\rho^2} - \frac{\kappa}{\rho} - \frac{1}{2} \right] W = 0 \\
W(0)=1, \qquad |W'(\rho)| \lesssim 1.
\end{cases}
\end{align*}
This is in turn equivalent to the fixed point formulation
$$
W(\rho) = 1 + \underbrace{\int_0^\rho s^{-1-d} e^{-\frac{s^2}{4}} \int_0^s t^{d+1} e^{\frac{t^2}{4}} 
\left[ \frac{d-1}{t^2} \cos (2\psi_\alpha) - \frac{d-1}{t^2} - \frac{\kappa}{t} - \frac{1}{2} \right] W(t) \,dt \,ds}_{T(W)}
$$
for $W \in L^\infty$. 
We see that $T$ satisfies  
\begin{align*}
\|TW\|_{L^{\infty}(0,\delta)} &\lec \delta 
\|W\|_{L^{\infty}(0,\delta)},
\\
\|TW-TW'\|_{L^{\infty}(0,\delta)} &\lec \delta 
\|W-W'\|_{L^{\infty}(0,\delta)},
\end{align*}
Therefore $T$ is a contraction on the ball $B(0,2)$ of $L^\infty(0,\delta)$, for $\delta$ sufficiently small.
Thus,  Banach's fixed point theorem gives the existence and uniqueness of $W$ in $L^\infty(0,\delta)$; it also gives smooth (analytic) dependence on $\kappa$ of $w(\rho)$ for any 
$\rho \in (0,\delta)$.

\bigskip
\noindent
\underline{Step 2: Prolonging $w$ and positivity.} The solution $w$ can be prolonged to $[0,\infty)$ by Lemma~\ref{lemmaODE}. Since the ODE is smooth away from zero, we obtain that $w(\rho)$ depends smoothly
on $\kappa$ for any $\rho >0$. By the uniform convergence estimate~\eqref{bdcv}, we get that $w(\infty)$ depends continuously on $\kappa$.
For $\kappa = 0$, $w = \ph_\alpha$ which is positive for any $\rho >0$, and such that $\ph_\alpha(\infty)>0$. 

We now use the continuous dependence of $w(\rho)$ (for any $\rho$) and $w(\infty)$ on $\kappa$ to find $\kappa_0$ such that:
for any $\kappa \in (0,\kappa_0)$, $w(\rho)>0$ for any $\rho \in (0,R_0)$ (where $R_0$ is defined in Corollary~\ref{loon}) as well as $w(\infty) >0$. Then, Corollary~\ref{loon} implies that $w(\rho)>0$ for any $\rho$.
\end{proof}

\subsection{A second variation on $\ph_\alpha$}
This last ODE result will be important for the asymptotic stability of the self-similar expanders. 
\begin{lemma}
\label{lemma:1}
For 
$$
\alpha \in 
\left\{
\begin{array}{ll}
(0,\alpha_*)  & \mbox{if $3 \leq d \leq 6$} \\
(0,\infty) & \mbox{if $d \geq 7$}
\end{array}
\right.
$$
 the problem
\begin{align}
\begin{cases}
\displaystyle{
y''+\left(\frac{d-1}\rho+\frac{\rho}2\right)y '-\frac{d-1}{\rho^2}
 \cos (2\psi_{\alpha})y+ \frac{y}{2}=0}, 
\\
\displaystyle{
y(0)=0, \qquad y '(0)=1,
}
\end{cases}
  \label{eq:l}
\end{align}
has a unique solution, which is positive for $\rho>0$ and satisfies
$\displaystyle{\lim_{\rho\rightarrow \infty}} \rho y(\rho)>0$. 
%where $\kappa =\frac12(\sqrt{d^2-4(d-1)\ve} -(d-2))$.
\end{lemma}
\begin{proof} Existence can be obtained by standard ODE techniques, for instance as in Step 1 of Lemma~\ref{perturbation}.

Turning to the statements that $y(\rho)>0$ and $\lim_{\rho \to \infty} \rho y(\rho) > 0$, the proof follows closely \cite[Appendix]{Naito}, but we nevertheless give its outline. To this end, we recall a set of lemmata from \cite{Naito}.

\begin{lemma}
{\bf \cite[Proposition B.1]{Naito}}
\label{lem:A.1}
Consider the differential equation:
\begin{align}
\label{eq:w}
(p_0(t)w')'+p_0(t)q_0(t)w=0
\end{align}
where $p_0(t)=\exp(C_0t+\frac14 e^{2t})$ (with $C_0 \in \R$ is a constant), and 
$q_0 \in C(\R) \cap L^\infty(\R)$.
Let $w \in C^2(\R)$ be a solution of \eqref{eq:w}.
Then there exists some $\ell \in \R$ such that 
$$
\lim_{t\rightarrow \infty} w(t)=\ell.
$$
Furthermore, $\ell=0$ if and only if 
$$
\int^\infty_{\tau}\frac{dt}{p_0(t)w^2(t)} =\infty
$$
for some $\tau \in \R$.
\end{lemma}

The next lemma is almost identical with 
\cite[Proposition A.1]{Naito} except the initial condition. 
Since it is proved in the same manner, we omit the proof.
\begin{lemma} 
\label{lem:A.2}

Consider the two equations:
\begin{align}
(p(r)u')'+q(r)u=0,
\label{eq:u}
\\
(p(r)v')'+Q(r)v=0,
\label{eq:v}
\end{align}
where $p$, $q$, $Q \in C([0,\infty))$, $p(r)>0$ for $r>0$,
and $Q(r) \ge q(r)$ on $[0,\infty)$ with $Q\equiv \!\!\!\!\!\!/ \ q$.

Assume that \eqref{eq:v} has a positive solution $v$ on $(0,\infty)$ such that 
$v(0)=0$, and $v'(0)$ exists.
Let $u$ be a solution of \eqref{eq:u} satisfying $u(0)=0$ and $u'(0)>0$.
Then $u$ is positive on $(0,\infty)$ and satisfies
$$
\int^\infty_R \frac{dr}{p(r)u^2(r)}<\infty
$$
for any $R>0$.
\end{lemma}

We now complete the proof of Lemma \ref{lemma:1}.
Since \eqref{ODE1} has a positive solution $\underline{\varphi}$, 
the comparison principle shows $y$ is also positive on 
$(0,\infty)$. Set $w(t)=\rho y(\rho)$ with $\rho=e^t$, then $w$ satisfies
\eqref{eq:w} with $C_0=d-4$ and 
$q_0(t)=-(d-1)\cos (2\psi(e^t))-d+3$. 

Lemma \ref{lem:A.1} yields that $w(t)$ converges 
to some $\ell \ge 0$ as $t\rightarrow \infty$; to get that $\ell>0$, it suffices to show that
\begin{equation}
\label{eq:integrable}
\int^\infty_{c}\frac{dt}{p_0(t)w^2(t)}<\infty.
\end{equation}
Now observe that $u(r)=y(r)$ and $v(r)=\underline{\varphi}(r)$ 
satisfy \eqref{eq:u} and 
\eqref{eq:v} respectively in Lemma \ref{lem:A.2} with 
$p(r)=r^{d-1}e^{\frac{r^2}{4}}$,  
$q(r)=-r^{d-3}e^{\frac{r^2}{4}}(\cos (2\psi(r))-\frac{r^2}{2})$ 
and $Q(r)=-r^{d-3}e^{\frac{r^2}{4}}(\cos (2\psi(r))-r^2)$. 
Since $Q(r)\ge q(r)$ on $[0,\infty)$ and $\underline{\varphi}$ 
 is positive on $(0,\infty)$, Lemma \ref{lem:A.2} shows
\begin{equation*}
\int^\infty_{R}\frac{d\rho}{p(\rho)y^2(\rho)}<\infty, 
\end{equation*}
which is nothing but \eqref{eq:integrable} upon changing the integration variable to $t=\log \rho$. 
\end{proof}

\section{Stability of self-similar expanders}

\label{oriole2}

In this section, we study the stability property of the self-similar expanders. 
In particular, it will be shown  that there exist two different expanders (originating at the north and south poles respectively)
evolving from the same initial data 
such that both are stable as stated in Theorem \ref{theostable}.

To this end, define first the set $E$ of values for which an expander 
as in Proposition~\ref{monotonicity} is available:
$$
E = \left\{ \begin{array}{ll}
(0,\ell^* = \psi_{\alpha^*}(\infty)) \setminus \{ \psi_{\alpha_i}(\infty), i=1 \dots N \} & \mbox{if $3 \leq d \leq 6$} \\
(0,\frac{\pi}{2}) & \mbox{if $d \geq 7$}
\end{array} \right.
$$
(notice in particular that, for $3 \leq d \leq 6$, $E \supset (0,\pi/2)$).
The main result of this section reads as follows.
\begin{theorem}
\label{thm:stability} 
Consider~\eqref{1} with data $h_0 \in L^\infty$ such that
\begin{equation}
\label{eq:h_0}
h_0(0)=\ell \in E, \quad |h_0(r)-h_0(0)| \lesssim r.
\end{equation}
Denote $\psi = \psi[\ell]$ for the profile of 
the self-similar expander satisfying $\psi(\infty)=\ell$ 
given by Proposition~\ref{monotonicity}.
\begin{itemize}
\item[(i)] (Small perturbations) If $\| h_0 - \ell \|_{L^\infty}$ is sufficiently small, then~\eqref{1} admits a global solution such that
for all $t,r \geq 0$,
\begin{equation}
\label{est:f}
\left| h(t,r) - \psi \left( \frac{r}{\sqrt t} \right) \right| \lesssim \| h_0 - l\|_{L^\infty}.
\end{equation}
\item[(ii)] (Large perturbations) Assuming that
$$
|\partial_r h_0(r)| \lesssim \frac{1}{r} \quad \mbox{and} \quad |\partial_r^2 h_0(r)| \lesssim \frac{1}{r^2},
$$
there exists $T>0$ and a solution $h \in L^\infty(0,T; L^\infty)$ to~\eqref{1} such that $|h(t,r)| \lesssim \min(1, \frac{r}{\sqrt{t}})$. It is close to $\psi$ in the following sense:  for all $\epsilon>0$, there exists $\zeta>0$ such that
\begin{equation}
\label{est:stability}
t + r < \zeta \quad \implies \quad \left| h(t,r) - \psi\left( \frac{r}{\sqrt t} \right) \right| < \epsilon.
\end{equation}
Furthermore, $h$ satisfies the local energy inequality~\eqref{EnIn}.
\end{itemize}
\end{theorem}
%Recalling that for $\ell \in [\pi/2-\delta_*, \pi/2]$
%there exist profiles satisfying $\psi_(\infty)=\ell$ by Theorem \ref{thm:profile}
%and the north or south boundary conditions respectively. 
%the expander satisfying 
%$\psi(\infty)$
Theorem \ref{theostable} follows from (ii). 
In section \ref{sec:asymptotic}, an asymptotic 
stability result for the expanders is also shown.

The proof of Theorem \ref{thm:stability} 
will rely on the perturbed initial value problem satisfied by $f = h - \psi \left( \frac{r}{\sqrt t} \right)$ if $h$ 
solves~\eqref{1}:
\begin{align}
\begin{cases}
\displaystyle{
f_t -f_{rr}-\frac{d-1}r f_r+\frac{d-1}{2r^2} 
\left[\sin \left(2\left(\psi\left(\frac{r}{\sqrt{t}}\right)+f\right)\right)
- \sin \left(2\psi\left(\frac{r}{\sqrt{t}}\right)\right)
\right] =0}
\\
\displaystyle f(t=0)=f_0 = h_0 - \ell.
\end{cases}
\label{P}
\end{align}
Notice that our assumptions on $h_0$ imply that $|f_0(r)| \lesssim r$ and $f_0 \in L^\infty$.

\subsection{Linear stability}

\label{swallow}

We start by proving estimates on the linearized problem at time 1:
\begin{align}
\begin{cases}
\displaystyle
f_t -f_{rr}-\frac{d-1}r f_r+\frac{d-1}{r^2} 
\cos \left(2 \psi\left(\frac{r}{\sqrt{t}}\right)\right) f = F 
\\
\displaystyle f(t=1)=f_0.
\end{cases}
\label{P}
\end{align}
In the self-similar variables
$$
s = \log t, \quad \rho = \frac{r}{\sqrt{t}}, \quad v(s,\rho) = f(e^s,e^{\frac{s}{2}} \rho) = f(t,r),
$$
it becomes
\begin{align}
\begin{cases}
v_s +H_{\alpha}v=F,
\\
v(s=0)=v_0
\end{cases}
\label{L}
\end{align}
with
\begin{equation}
\label{defH}
H_{\alpha}v=-v_{\rho\rho}- \left( \frac{d-1}\rho+\frac{\rho}{2} \right) v_{\rho}
+V_\alpha v \qquad \mbox{where} \ V_\alpha = \frac{d-1}{\rho^2}\cos (2\psi_\alpha).
\end{equation}
Define further the functional space $L^\infty[w]$ by its norm
$$
\| f \|_{L^\infty[w]} = \left\| \frac{f}{w} \right\|_{L^\infty}.
$$
\underline{Here and in the following, $w$ is always as in Lemma~\ref{perturbation}, associated to some fixed $\kappa \in (0, \kappa_0)$}.
\begin{lemma}
\label{linear}
Assume 
$$
\alpha \in 
\left\{ \begin{array}{ll}
(0,\alpha^*) \setminus \{ \alpha_i, \, i =1 \dots N \} & \mbox{if $3 \leq d \leq 6$} \\ (0,\infty) & \mbox{if $d \geq 7$}.
\end{array} \right.
$$
Assume furthermore that $v_0\in L^\I_\rho[w]$, $Z F \in L^{\infty}_s L^{\infty}_\rho[w]$. Then there exists a unique solution in $L^\I_s L^\I_\rho[w]$ to~\eqref{L}
 satisfying the estimate
\begin{equation}
\label{est:v}
\|v\|_{L^\I _s L^\I_\rho [w]} 
\lec \|v_0\|_{L^\I_\rho [w]} + \|Z F\|_{L^\I_s L^\I_\rho [w]}.
\end{equation}
\end{lemma}
\begin{proof}
\underline{Step 1: Existence in $L^\infty_{s,\rho}$.} In order to prove the existence of a solution to~\eqref{L} in $L^\infty_{s,\rho}$, consider the regularized problem
\begin{align*}
\begin{cases}
 v^\epsilon_s +H_{\alpha}^\epsilon v^\epsilon = F,
\\
v^\epsilon(s=0)=v_0,
\end{cases}
\end{align*}
where $H_{\alpha}^\epsilon v = -v_{\rho\rho}- \left( \frac{d-1}\rho+\frac{\rho}{2} \right) v_{\rho}
+ \frac{d-1}{\epsilon + \rho^2}\cos (2\psi_\alpha)v$.
Without loss of generality, we may assume $v_0\ge 0$ 
and $F \ge 0$.  
By noticing that $v_0 \in L^\infty_\rho$ and $F \in L^\infty_{s,\rho}$,
it is easy to see that this problem has global solutions (indeed, coming back to the original variables, this is nothing but a heat equation with a bounded potential). 
Moreover these solutions are uniformly bounded 
in $\ve>0$ in $L^\infty_{s, \rho} ([0,T] \times \mathbb{R}_+)$ for each $T>0$, 
that is, there exists $C_T>0$ such that 
\begin{equation}
\label{est:unif}
\|v^{\delta}\|_{L^\infty([0,T]\times \R_+)} \le C_T.
\end{equation}
Indeed we see from the equation that 
\begin{align*}
v^{\ve}_s&=-H^{\ve}_{\alpha}v^{\ve} +F
\\
&\le 
v^{\ve}_{\rho \rho} +  \left( \frac{d-1}\rho+\frac{\rho}{2} \right) v^{\ve}_{\rho}
+ \frac{d-1}{\rho_0^2} v^{\ve} +F,
\end{align*}
where $\rho_0>0$ is the minimal zero of the function 
$\cos (2\psi_\alpha)$. By the maximum principle, we have
$$
\|v^{\ve}(s)\|_{L^{\infty}_{\rho}} \le e^{\frac{d-1}{\rho_0^2}}
\|v^{\ve}_0\|_{L^{\infty}_{\rho}} +\|F\|_{L^{\infty}_{s,\rho}}T,
$$
which shows  \eqref{est:unif}.

Switching back to the original variables $(t,r)$, with $v^\epsilon(s,\rho) = f^\epsilon(e^s,e^{s/2}\rho) = f^\epsilon(t,r)$, 
and $f^\epsilon(t=1,r) = v^\epsilon(s=0,r)$, the equation solved by $f^\epsilon$ can be written in Duhamel form as
$$
f^\epsilon = e^{(t-1) \Delta} v_0 + \int_1^t e^{(t-t') \Delta} \biggl[F^\epsilon-\frac{d-1}{\epsilon+r^2} \cos \left(2 \psi \left( \frac{r}{\sqrt t'} \right) \right)\ f^\epsilon\biggr]dt'.
$$
Observe that $F^\epsilon -\frac{d-1}{\epsilon+r^2} \cos \left(2 \psi \left( \frac{r}{\sqrt t'} \right) \right)f^\epsilon$ is locally bounded in $L^p_t L^q_x$ for all $p<\infty$ and $q<\frac{d}{2}$, uniformly in $\epsilon>0$, as well as bounded on the domain $\{ r > 1 \}$. By maximal regularity of the Laplacian (see for instance Lemari\'e-Rieusset~\cite{LR}, Theorem 7.3), this gives for any $1 \leq T_1 < T_2$, $p<\infty$ and $q < \frac{d}{2}$,
bounds on the derivatives of $f^\epsilon$ which are uniform in $\epsilon$:
$$
\| \partial_t f^\epsilon \|_{L^p(T_1,T_2,L^q_{\operatorname{loc}})} + \| \Delta f^\epsilon \|_{L^p(T_1,T_2,L^q_{\operatorname{loc}})} \lesssim 1. 
$$
Using these bounds and the Rellich-Kondrakov embedding theorem, it is not hard to pass to the limit $\epsilon \to 0$ 
and get the existence of a solution $f(t,r)$, hence $v(s,\rho)$ of~\eqref{L}. 
By \eqref{est:unif}, this leaves us with a solution $f \in L^\infty_t L^\infty_x$, such that $\partial_t f$ and $\Delta f$ belong locally to $L^p_t L^q_x$, for $p<\infty$, $q<\frac{d}{2}$.

Observe that the above estimates actually show that any solution in $L^\infty_t L^\infty_x$ is such that $\partial_t f$ and $\Delta f$ belong locally to $L^p_t L^q_x$, for $p<\infty$, $q<\frac{d}{2}$.

\bigskip

\noindent \underline{Step 2: Maximum principle.}~
We prove that 
the maximum principle holds: 
we assume that $v_0 \leq 0$ and $(\partial_s + H_\alpha) v \leq 0$, and that $v\in L^\infty_t L^\infty_x$, with $\partial_t v$ and $\Delta v$ locally in $L^p_t L^q_x$ for $p<\infty$ and $q<\frac{d}{2}$; and aim at showing that $v(s) \leq 0$ for any $s \geq 0$.

We will use Stampacchia's method by bounding the $L^2$ norm of $v_+ = \max(0,v)$. 
Assume for the moment that there exists $R>0$ such that $v_+$ is supported on $B(0,R)$. Proceeding as in Sohr~\cite{Sohr} for example, an energy estimate gives
$$
\frac{1}{2} \frac{d}{ds} \int_{\mathbb{R}_+} |v_+|^2 \,\rho^{d-1} \, d \rho 
\le \int_{\mathbb{R}_+} v_+ \left[ \partial_\rho^2 + \frac{d-1}{\rho} \partial_\rho + \frac{\rho}{2} \partial_\rho - V_\alpha \right] v \,\rho^{d-1} \, d \rho 
= I + II + III.
$$
By the space-time bounds on $v$, it is not hard to justify standard integrations by parts which give
\begin{align*}
& I =  - \int_{\mathbb{R}_+} |\partial_\rho v_+|^2\,\rho^{d-1} \, d \rho \\
& II = - \frac{d}{4} \int_{\mathbb{R}_+} |v_+|^2\,\rho^{d-1} \, d \rho \\
& III = - \int_{\mathbb{R}_+} V_\alpha |v_+|^2\,\rho^{d-1} \, d\rho \lesssim \int_{\mathbb{R}_+} |v_+|^2\,\rho^{d-1} \, d \rho
\end{align*}
(using in the last inequality that $V_\alpha$ is bounded from below). This implies the differential inequality
$$
\frac{d}{ds} \int_{\mathbb{R}_+} |v_+|^2 \,\rho^{d-1} \, d \rho \lesssim \int_{\mathbb{R}_+} |v_+|^2 \,\rho^{d-1} \, d \rho,
$$
which, combined with the fact that $ \int_{\mathbb{R}_+} |v_+|^2 \,\rho^{d-1} \, d \rho = 0$ at $s=0$, results in the identity $v_+ \equiv 0$. 

This is the desired result, but it was derived under a compact support assumption on $v_+$.

In order to get rid of this assumption, introduce $z = v - \epsilon( Ms + \log \langle \rho \rangle)$. Observe that it satisfies the compact support property for all $\epsilon>0$; choosing $M$ sufficiently big, it satisfies furthermore for all $\epsilon>0$, and for $s$ sufficiently small
$$
(\partial_s + H_\alpha) z \leq 0.
$$
The same approach as for $v$ can be applied to show $z_+ =0$ for $s$ sufficiently small, and then, iterating, for $s>0$.

\bigskip

\noindent \underline{Step 3: Inhomogeneous estimate.} For $v$ as in the statement of the lemma, define
$$
z := v - (\|v_0\|_{L^\infty[w]}  + \frac{\| Z F \|_{L^\infty_s L^\infty_\rho[w]}}{\kappa}) w.
$$ 
Obviously, $z(s=0) \leq 0$. Furthermore, by the definition of $w$, $H_\alpha w = \frac{\kappa}{Z(\rho)} w$, therefore
$$
(\partial_s + H_\alpha) z \leq (\partial_s + H_\alpha) v - \| Z F \|_{L^\infty_s L^\infty_\rho[w]} \frac{w}{Z} \leq 0.
$$
Therefore we have $z(s) \leq 0$ for $s>0$. 
Similary for $z':=v+ (\|v_0\|_{L^\infty[w]}  + \frac{\| Z F \|_{L^\infty_s L^\infty_\rho[w]}}{\kappa}) w$, if $z'(s=0) \ge 0$ then
$z'(s)\ge 0$ for $s>0$, which is the desired  estimate \eqref{est:v}, and also implies the uniqueness.
\end{proof}

\subsection{Proof of Theorem \ref{thm:stability} $(i)$: stability for small perturbations} 
\label{proofsmall}

\ 

\noindent
\underline{Step 1: Solvability from time $t=\delta$.}
 We consider the perturbation from the self-similar expander
$f=h-\psi(\cdot/\sqrt{t})$.
For given data at $t=\delta$, the initial value problem \eqref{1} becomes
\begin{align}
\begin{cases}
\displaystyle{
f_t -f_{rr}-\frac{d-1}r f_r+\frac{d-1}{2r^2} 
\left[\sin \left(2\left(\psi\left(\frac{r}{\sqrt{t}}\right)+f\right)\right) - \sin \left(2\psi\left(\frac{r}{\sqrt{t}}\right) \right) \right] =0}, \\
f(t=\delta)=f_0.
\end{cases}
\label{P}
\end{align}
We will prove that this problem is globally well posed, uniformly in $\delta> 0$.
More precisely, we want to show that there exists $\ve_1 >0$ such that: if $\| f_0 \|_{L^\infty[w(\cdot/\sqrt \delta)]} \le \ve_1$,
then it admits a unique solution such that
\begin{align}
\label{est:f}
\|f\|_{L^{\infty}(\delta,\infty;L^{\infty}[w(\cdot/\sqrt{t})])}
\lesssim \|f_0\|_{L^\infty[w(\cdot/\sqrt\delta)]}.
\end{align}
In order to prove this assertion, we use the self-similar variables $s=\log t$, $\rho=\frac{r}{\sqrt{t}}$, and set $v(s,\rho)=f(e^s,e^{s/2}\rho)=f(t,r)$.
Then~\eqref{P} becomes 
\begin{align*}
\begin{cases}
\displaystyle{
 v_s - v_{\rho \rho}-(\frac{d-1}\rho+\frac{\rho}{2} )v_{\rho}+
\frac{d-1}{2\rho^2} 
[\sin (2(\psi+v))- \sin (2\psi)]
=0,}
\\
v(s=\log \delta)=f_0(\sqrt{\delta}\,\cdot).
\end{cases}
\end{align*}
or
\begin{align}
\label{S2}
\begin{cases}
\displaystyle{ v_s +H_{\alpha}v+
\frac{J(v)}{\rho^2} 
=0,} \\
v(s=\log \delta)=f_0(\sqrt{\delta}\,\cdot),
\end{cases}
\end{align}
where $H_\alpha$ is defined in~\eqref{defH} and
\begin{equation}
\label{defJ}
J(v)=\frac{d-1}{2}[\sin (2(\psi+v))- \sin (2\psi)-2\cos(2\psi)v].
\end{equation}
(notice that it satisfies $|J(v)| \lesssim |v|^2$).
The problem is reduced to that of showing the existence of 
$v$ such that 
\begin{align}
\|v\|_{L^{\infty}(\log \delta,\infty;L^{\infty}[w])} \lesssim \|f_0(\sqrt{\delta}\,\cdot)\|_{L^\infty[w]}.
\end{align}
It is now easy to solve \eqref{S2} by a standard contraction argument. Indeed by Duhamel's formula, it can be written
$$
v(s)=e^{-sH_\alpha}f_0(\sqrt{\delta}\,\cdot)+\int^s_{\log \delta} e^{(\sigma-s)H_\alpha}
\frac{J(v(\rho,\sigma))}{\rho^2}d\sigma=:B(v).
$$
Lemma \ref{linear}, along with the bound $|J(u)| \lesssim |u|^2$, gives the estimate
\begin{align*}
\|B(v)\|_{L^\I _s L^\I_\rho [w]} 
&\lec \|f_0(\sqrt{\delta}\,\cdot)\|_{L^\I [w]} + \left\|\frac{Z}{\rho^2} v^2\right\|_{L^\I_s L^\I_\rho [w]} \\
&\lec \|f_0(\sqrt{\delta}\,\cdot)\|_{L^\I [w]} + \|v\|_{L^\I_s L^\I_\rho[w]}^2,
\end{align*}
where we have used the fact $Zw \lesssim \rho^2$ in the second line.
Similarly we have 
$$\|B(v)-B(\tilde{v})\|_{L^\I _s L^\I_\rho [w]} 
\lec
{\rm max}(\|v\|_{L^\I_s L^\I_\rho [w]}, \|\tilde{v}\|_{L^\I_s L^\I_\rho [w]})
 \|v-\tilde{v}\|_{L^\I_s L^\I_\rho [w]}.
$$
Thus the map $B$ is a contraction on a small ball in 
$L^\I_t L^\I [w]$, which gives the desired result.

\bigskip
\noindent
\underline{Step 2: Limiting procedure as $\delta \to 0$}
By the assumption \eqref{eq:h_0} and the positivity of $w$, there exists 
$\delta_0>0$ such that for $\delta <\delta_0$, $\|f_0\|_{L^\infty[w(\cdot/\sqrt{\delta})]} \le \ve_1$ holds.
Hence, Step 1 gives a sequence of solutions $f^{\delta}$ of~\eqref{P} with data $f_0$ at time $\delta$: for any $\varphi \in \mathcal{C}^\infty ([0,\infty) \times \mathbb{R}^d)$, 
\begin{equation}
\begin{split}
\label{bird}
& \int_{\delta}^\infty \int f^{\delta} \left( -\varphi_t- \Delta \varphi \right)\,dx\,dt \\
& \qquad + \int_{\delta}^\infty \int  \frac{d-1}{2|x|^2}
\left( \sin(2(\psi(|x|/\sqrt{t})+f^{\delta})) - \sin(2 \psi(|x|/\sqrt{t})) \right)
\varphi \,dx\,dt \\
& \qquad \qquad = \int f_0 \, \varphi(t=\delta) \,dx.
\end{split}
\end{equation}
Furthermore,
\begin{align}
\label{est:delta}
\|f^{\delta}\|_{L^{\infty}(\delta,\infty;L^{\infty}[w(\cdot/\sqrt{t})])}
\lesssim \|f_0\|_{L^\infty[w(\cdot/\sqrt\delta)]}\le C\ve_1.
\end{align}
Arguing as in Section~\ref{swallow}, we obtain by maximal regularity of the Laplacian, for any $0 < T_1 < T_2$, $p<\infty$ and $q < \frac{d}{2}$,
bounds on the derivatives of $f^\delta$ which are uniform in $\delta$:
$$
\| \partial_t f^\delta \|_{L^p(T_1,T_2,L^q)} + \| \Delta f^\delta \|_{L^p(T_1,T_2,L^q)} \lesssim 1. 
$$
By the above bounds and the  Banach-Alaoglu and Rellich-Kondrakov theorems, we can find a sequence $(\delta_n)$ going to zero such that 
$f^{\delta_n}$ converges weakly locally in $L^p_{t,x}$ to some $f $ for all finite $p$; 
and furthermore $f^{\delta_n}$ converges strongly in $L^1$ on compact sets avoiding $\{t=0\}$.

Here $f $ is bounded since 
$$
 |f^{\delta}(t,r)| \le C\ve_1 w(\frac{r}{\sqrt{t}}) \le C \ve_1 \|w\|_{\infty}
$$
holds uniformly in $\delta>0$.
It is then easy to pass to the limit in~\eqref{bird} to see that $f$ satisfies
\begin{equation}
\begin{split}
\label{hummingbird}
& \int_{0}^\infty \int f \left( -\varphi_t- \Delta \varphi \right)\,dx\,dt \\
& \qquad + \int_0^\infty \int  \frac{d-1}{2|x|^2}
\left( \sin(2(\psi(|x|/\sqrt{t})+f)) - \sin(2 \psi(|x|/\sqrt{t})) \right)
\varphi \,dx\,dt \\
& \qquad \qquad = \int f_0 \, \varphi(t=0) \,dx.
\end{split}
\end{equation}
Finally since there exists a constant $C>0$ independent of $\delta$ such that 
$w(r/\sqrt{t}) \le Cw(r/\sqrt{\delta})$ for $t>\delta$, 
we see from \eqref{est:delta} that 
$$
|f^{\delta}(t,r)| \le |f_0(r)|\frac{w(\frac{r}{\sqrt{t}})}{w(\frac{r}{\sqrt{\delta}})} \le C|f_0(r)|,
$$
which yields \eqref{est:f}.

\subsection{Construction of sub- and super-solutions}
In order to deal with large perturbations, we construct 
local-in-time super- and subsolutions.
\begin{proposition}
\label{supersolution}
Consider the equation for the perturbation in self-similar variables:
\begin{equation}
\label{pinguin}
u_s + H_\alpha u = - \frac{J(u)}{\rho^2},
\end{equation}
(see~\eqref{S2}).
Let $M>0$, $R>0$, and $\phi$ a smooth non decreasing function such that
\begin{equation}
\label{def:cutoff}
0 \leq \phi \leq 1 \quad \mbox{and} \quad \left\{ \begin{array}{ll} \phi = 0 & \mbox{on $[0,R]$} \\ \phi>0 & \mbox{on $(R,\infty)$} \\ \phi = 1 & \mbox{on $[2R,\infty)$} \end{array} \right.
\end{equation}
Then there exist $\eta_0>0$, $A>0$, $s_0 \in \mathbb{R}$ such that for $\eta \in (0,\eta_0)$, the functions
\begin{align*}
u_+(s,\rho) & = \eta w(\rho) + \left( M + \frac{A}{\rho^2} \right) \phi(e^{s/2}\rho),  \\
u_-(s,\rho) & = - u_+ 
\end{align*}
are super- and subsolutions respectively of~(\ref{pinguin}), namely
\begin{align}
\label{hummingbird} & (u_+)_s + H_\alpha u_+ \geq -\frac{J(u_+)}{\rho^2}, 
\\
& (u_-)_s + H_\alpha u_- \leq -\frac{J(u_-)}{\rho^2} 
\end{align}
for $-\infty < s < s_0$.
\end{proposition}

\begin{proof}
We only prove here that $u = u_+$ is a supersolution, since the fact that $u_-$ is a subsolution can be proved identically. 
First examine the left-hand side of~(\ref{pinguin}). 
Decompose first
\begin{align*}
(\partial_s + H_\alpha) u_+ 
& = \underbrace{\eta H_\alpha w}_{I} + \underbrace{\phi (e^{s/2} \rho) H_\alpha \left( M + \frac{A}{\rho^2} \right)}_{II} \\
& \qquad + \underbrace{\left( M + \frac{A}{\rho^2} \right) \left( - e^s \phi''(e^{s/2} \rho) - \frac{(d-1) e^{s/2}}{\rho} \phi'(e^{s/2}\rho) \right) + \frac{4A}{\rho^3} e^{s/2} \phi'(e^{s/2}\rho)}_{III}.
\end{align*}
Lemma~\ref{perturbation} gives
$$
I = \frac{\kappa \eta}{Z} w.
$$
Recall that $r = e^{s/2}\rho$ and introduce for convenience the notation 
$$
V = \frac{d-1}{\rho^2} \cos(2 \psi_\alpha).
$$
A computation gives
$$
II = \phi(r) \left( V(\rho)M + V(\rho) \frac{A}{\rho^2} + \frac{A}{\rho^2} + \frac{2(d-4)A}{\rho^4} \right).
$$
Denote $\ell \in \mathbb{R}$ for the limit at infinity of $\rho^2V(\rho)$; we know that $V(\rho) = \frac{\ell}{\rho^
2} + O \left(\frac{1}{\rho^3}\right)$.
Due to the support property of $\phi$, this implies 
$$
II = \frac{\phi(r)}{\rho^2} \left( M \ell + A + (M+A) O\left( \frac{e^{s_0/2}}{R} \right)  \right). 
$$
Finally, since $\phi'$ and $A$ are nonnegative,
$$
III \geq e^s \left( M + \frac{A}{\rho^2} \right) \left( -\phi''(r) - \frac{(d-1)\phi'(r)}{r} \right).
$$

On the other hand,  it is easy from the definition of $J(u)$ in \eqref{defJ} 
to see that
$$
\left| J (u) \right| \leq C_1  \inf \left( 1, u^2 \right)
$$
for a constant $C_1$.

We now compare both sides of~\eqref{pinguin} by distinguishing between three regions.

\medskip
\noindent
\underline{If $\rho \leq R e^{-s/2}$}, $\phi(r)=0$, so that~(\ref{hummingbird}) is satisfied if 
$\displaystyle I = \frac{\eta \kappa w}{Z} \geq C_1 \frac{u_+^2}{\rho^2} = C_1 \frac{\eta^2 w^2}{\rho^2}$, 
which is implied by
$$
\eta < \frac{\kappa}{C_1} \inf_{\rho>0} \frac{\rho^2}{Zw}
$$
(notice that $ \inf_{\rho>0} \frac{\rho^2}{Zw}>0$), which can be achieved by choosing $\eta_0$ sufficiently small.

\medskip
\noindent
\underline{If $\rho \geq 2R e^{-s/2}$}, $\phi(r)=1$, so that~(\ref{hummingbird}) is satisfied if
$\displaystyle II \geq \frac{C_1}{\rho^2}$ or in other words
$$
M\ell + A + (M+A) O \left( \frac{e^{s_0/2}}{R} \right) \geq C_1,
$$
which can be guaranteed by first taking $A$ sufficiently large, and then $s_0$ sufficiently close to~$-\infty$.

\medskip
\noindent
\underline{If $\rho \in (R e^{-s/2}, 2R e^{-s/2})$}, none of $I$, $II$ and $III$ vanishes. First, we restrict $s_0$ to be so close
to $-\infty$ that
$$
\mbox{in $II$}, \;\;\left|(M+A)O\left( \frac{e^{s_0}}{R} \right) \right| < 1 \qquad \mbox{and} \qquad \frac{A}{\rho^2} < M
$$
(so that $s_0$ depends on $A$, $M$, and $R$). Then,~(\ref{hummingbird}) is satisfied if
$$
\frac{\kappa \eta \rho^2}{Z} w + \phi(r)(M\ell + A-1) - 8 R^2 M \left( |\phi''(r)| + (d-1) \frac{|\phi'(r)|}{|r|} \right) 
\geq C_1 \inf \left( 1,\left(\eta w + 2M \phi(r)\right)^2 \right).
$$
We can assume that $s_0$ is so close to $-\infty$ that $\rho>1$ and $|w(\rho) - w(\infty)| < \frac{1}{2} w(\infty)$ for $\rho e^{s_0/2} \geq R$. Then, assuming furthermore that $A$ is so big that the second summand above is nonnegative, it suffices to show that
$$
 \frac{\kappa \eta \rho^2}{2Z} w(\infty) - 8R^2M \left( |\phi''(r)| + (d-1) \frac{|\phi'(r)|}{|r|} \right) \geq C_1 \inf(1,(\frac 3 2 w(\infty) \eta + 2M\phi)^2 ). 
$$
This is true with a strict inequality for $r=R$ by requiring that
$$
\eta < \frac{\kappa}{10 C_1 w(\infty)} \inf_{\rho>1} \frac{\rho^2}{Z}.
$$ 
By continuity, this remains true for $r  \in (R, R+\delta)$, with $\delta>0$ and $\phi(R+\delta)>0$. 
Thus we only need to ensure that~(\ref{hummingbird}) holds if $r \in (R+\delta, 2R)$. For this to 
be true, it suffices that for $r \in (R+\delta, 2R)$,
$$
\phi(r)(M\ell + A-1) - 8 R^2 M \left( |\phi''(r)| + 2 \frac{|\phi'(r)|}{|r|} \right) \geq C_1.
$$
Using the fact that $\inf_{(R+\delta,2R)} \phi > 0$, this last inequality is ensured by taking $A$ even bigger if necessary.
\end{proof}

\subsection{Proof of Theorem \ref{thm:stability}$(ii)$: stability under large perturbations}

\

\noindent
\underline{Step 1: Solvability from time $t=\delta$.}
As in the proof of $(i)$ (Section~\ref{proofsmall}), 
we consider the problem \eqref{bird} 
with data $ f_0(r)=h_0(r)-l$ priscribed at time $t=\delta>0$.
%or equivalently the full problem with data $f_0 + \psi \left( \frac{\cdot}{\sqrt{\delta}} \right)=h_\delta$.

Since $h_\delta/r$ is bounded,  \cite[Proposition III.1]{BS} can apply to show the existence of 
a smooth solution $f^\delta$ on $( \delta, T_{\delta})$ from the data $f(t=\delta)=f_0$ 
with some $T_{\delta} > \delta$. 

To prolong these solutions $f^\delta$ until some time $T_0>\delta$ uniformly in $\delta>0$, 
we switch to self-similar coordinates $(s,\rho)$ and will apply the extension criterion 
Proposition \ref{proposition:holder}. To this end, we will use the comparison with sub- and supersolutions.

For $\eta < \eta_0$ we apply Proposition \ref{supersolution} with some $R$, $M>0$ 
to be chosen later, which gives a subsolution $u_-$ and a supersolution $u_+ = - u_-$, 
both defined on $[0,s_0]$. We will show that there exist positive constants $R$, $M$, $\delta_0$ 
such that for any $\delta \in (0,\delta_0)$, 
$$
|f_0(\sqrt \delta \rho)| \le \eta w(\rho) + \left(M+\frac{A }{\rho^2}\right) \phi(\sqrt{\delta}\rho)
$$
holds, or equivalently, 
\begin{align}
\label{est:f_0}
|f_0(r)| \le \eta w\left(\frac{r}{\sqrt{\delta}}\right) + \left(M+\frac{\delta A }{r^2}\right) \phi(r).
\end{align}
Let  $M=\|f_0\|_{L^\infty}$, then   \eqref{est:f_0}
 holds for $r\ge 2R$,
because $M\phi(r)=\|f_0\|_{L^\infty}$. 
To see that \eqref{est:f_0} also holds when $r< 2R$, 
we set $\underline{w}(r):=\inf\{w(s)| 0<s<r\}$.
Then $\underline{w}(r)$ is monotone nondecreasing and satisfies
$\underline{w}(r) \le w(r)$. Moreover 
Lemma \ref{perturbation} shows that 
$\underline{w}(r) \gec r$ for small $r>0$ and 
$\underline{w}(r) \rightarrow w(\infty)>0$
as $ r \rightarrow \infty$. 
Therefore, choosing $R=\frac12 \inf \{\rho \,|\, f_0(\rho)=
\eta \frac{w(\infty)}{2}\}$
we can find  $\delta_0>0$ such that 
$$
|f_0(r)| \le \eta \underline{w}\left(\frac{r}{\sqrt{\delta_0}}\right) 
$$
holds for $r<2R$. By monotonicity of $\underline{w}$, 
$$
\eta \underline{w}\left(\frac{r}{\sqrt{\delta_0}}\right)  
\le 
\eta \underline{w}\left(\frac{r}{\sqrt{\delta}}\right) 
\le
\eta w\left(\frac{r}{\sqrt{\delta}}\right),
$$
which shows \eqref{est:f_0} for $r <2R$.

By the comparison principle (Lemma \ref{lemma:comparison}), we obtain that 
\begin{equation}
\label{est:f(t)}
|f^{\delta}(t,r)| \le u_+ \left(\log t,\frac{r}{\sqrt t} \right)
\end{equation}
for $ \delta<t<e^{s_0}$ as long as $f^{\delta}$ is smooth.

This bound along with the parabolic regularity result in Proposition~\ref{parabolicreg} 
allow us to apply the extension criterion in Proposition \ref{proposition:holder}, which shows that $f^\delta(t)$ can be prolonged until a time $t=T_0$ uniformly in $\delta$. 

\bigskip

\noindent
\underline{Step 2: Limiting procedure as $\delta \to 0$.} The previous step gives a sequence of solutions $(f_\delta)$ defined on $(\delta,T_0)$.
Arguing as in Step 2 of Section~\ref{proofsmall}, 
we obtain the limit function $f$ satisfying \eqref{hummingbird} 
for any $\varphi \in \mathcal{C}([0,T_0)\times \R^d)$, 
and hence $h(t,r) = f(t,r)-\psi(\frac{r}{\sqrt{t}})$ 
is the solution of \eqref{1}.

\bigskip
\noindent
\underline{Step 3: Bounds on $h$.}
To see that $|h(t,r)| \lesssim \min(1,\frac{r}{\sqrt{t}})$, 
it suffices to show $|f^{\delta}(t,r)| \lesssim \min(1,\frac{r}{\sqrt{t}})$ holds uniformly in $\delta$. Indeed 
\eqref{est:f(t)} shows 
\begin{align*}
|f^{\delta}(t,r)| 
&\le 
\eta w\left(\frac{r}{\sqrt{t}}\right) +\left(M+\frac{tA}{r^2}\right)\phi(r)
\\
&\le 
\eta w\left(\frac{r}{\sqrt{t}}\right) +\left(M+\frac{T_0A}{4R^2}\right)\phi(r)
\\
&\lec
\min\left(1,\frac{r}{\sqrt{t}}\right).
\end{align*}
Similarly, the above estimate gives that $|h - \psi ( \frac{r}{\sqrt t})| < \eta \|w\|_{L^{\infty}}$ for $r <R$. 
Since $\eta <\eta_0$ is arbitrary, this
gives the  desired bound \eqref{est:stability}.

\bigskip
\noindent
\underline{Step 4: Local energy inequality.} To prove that $h$ 
satisfies the local energy inequality, we first record that $|h(r,t)| \lesssim \frac{r}{\sqrt{t}}$. By our assumptions on $h_0$, Proposition~\ref{parabolicreg} gives the bound
$$
|\partial_t h(t,r)| \lesssim \frac{1}{r^2}.
$$
Furthermore, Propositions~\ref{parabolicreg3} gives
$$
|\partial_r h(t,r)| \lesssim \frac{1}{r + \sqrt t} \left( 1 + \langle \log \frac r {\sqrt t} \rangle \mathbf{1}_{r < \sqrt t} \right).
$$
These estimates enable us to use Proposition~\ref{energyinequality} to deduce that $h$ satisfies the local energy inequality.

\subsection{Convergence to the expanders for smooth data}
\label{sec:asymptotic}
In this section we prove the convergence to the expanders for smooth data close to a constant at $\infty$.
This result will play a key role in the uniqueness proof in the next section.

\begin{theorem}
\label{theorem:asymptotic}
Let $\ell \in [0,\pi/2)$ and  assume that 
$h_0 \in C([0,\infty))$ satisfies 
\begin{equation}
0 < h_0(r)\le \min\{ \frac{\pi}{2}-\delta,~  C_1r\} \qquad \textit{for}\ r>0 
\end{equation}
and 
\begin{equation}
|h_0(r)-\ell| \le C_2r^{-1} \qquad 
\textit{for}\ r\ge R
\label{snake}
\end{equation}
for some positive constants $\delta$, $C_1$, $C_2$ and $R>0$. 

Then there exist $\tilde{t}=\tilde{t}(\delta, C_1)>0$, 
$\tilde{C}=\tilde{C}(\delta,C_1,C_2,R)>0$ 
and 
the solution of \eqref{1} 
satisfying 
\begin{align}
\displaystyle{\left\| h(t) -
\psi_{N} \left(\frac{\cdot}{\sqrt{t+\tilde{t}}}\right)
\right\|_{L^\infty}
\le \tilde{C}\tilde{t}^{\frac12} (t+\tilde{t})^{-\frac12}} \qquad \textit{for} \ t>0, 
\end{align}
 where $\psi_N=\psi_N[\ell]$ is 
the self-similar profile given in Theorem \ref{thm:profile}.
\end{theorem}
\begin{remark}
A similar stability result was proved in Fila-Winkler-Yanagida~\cite{FWY} 
for the nonlinear heat equation.
\end{remark}
In order to prove stability, we will also use the self-similar variables:
$$
u(s,\rho)=h(e^s-\tilde t, e^{s/2}\rho)=h(t,r)
$$ for $\tilde{t}>0$.
Then $u$ satisfies the equation
\begin{align}
\displaystyle{
u_s 
- u_{\rho \rho}
-\left(\frac{d-1}\rho+\frac{\rho}{2} \right)u_{\rho}
+
\frac{d-1}{2\rho^2} \sin (2u) 
=0} \label{u}
\end{align} 
 for $s > \log \tilde t$, with the initial condition 
$u(\log \tilde t, \rho)=h_0(\sqrt{\tilde t}\rho)$.
As in \cite{FWY}, the following 
super- and  subsolutions for \eqref{u} play a key role in the proof:

\begin{lemma}

\label{lemma:u^+} Let $\alpha_0$ be the constant given in 
\eqref{defalpha0}.
For $\alpha \in (0, \alpha_0)$, let $\psi_\alpha$
be the expander satisfying the north pole boundary condition \eqref{northpole} given in Theorem \ref{thm:profile}(i).
{Let $V$ and $W$ be the solutions of 
\begin{align*}
&\begin{cases}
\displaystyle{V_{\rho\rho}+\left(\frac{d-1}\rho+\frac{\rho}2\right)
V_{\rho}-\frac{d-1}{\rho^2} \cos (2\psi_{\alpha_0})V+\frac{V}{2}=0
\qquad \rm{in} \ \R_+,} \\
V(0)=0, \qquad V_{\rho}(0)=1, 
\end{cases}
%\label{V}
\\
&\begin{cases}
\displaystyle{W_{\rho\rho}+\left(\frac{d-1}\rho+\frac{\rho}2\right)
W_{\rho}-\frac{d-1}{\rho^2} \cos (2\psi_{\alpha})W+\frac{W}{2}=0
\qquad \rm{in} \ \R_+,} \\
W(0)=0, \qquad W_{\rho}(0)=1, 
\end{cases}
%\label{W}
\end{align*}
respectively~(see Lemma~\ref{lemma:1}). 
Then for any $b>0$,
\begin{align*}
u^+(s,\rho)&:=\min 
\{\psi_{\alpha_0}(\rho), \psi_{\alpha}(\rho)+be^{-\frac{s}2}V(\rho) \},
\\
u^-(s,\rho)&:=\max 
\{0, \psi_{\alpha}(\rho)-be^{-\frac{s}2}W(\rho) \}
\end{align*}
are a supersolution and a subsolution of \eqref{u} respectively.
}
\end{lemma}

\begin{proof}
Since both are proved in the same way, we only prove that 
$u^+$ is a supersolution.
It suffices to check the case when $\rho$ satisfies
{
$\psi_{\alpha_0}(\rho) \ge \psi_{\al}(\rho) + be^{-\frac{s}{2}}V(\rho)$.  
%$u^+=\psi_{\alpha}+be^{-\frac{ls}{2}}W$.
From the fact that $\psi_{\alpha_0} <\pi/2$ and the 
concavity of $x \mapsto \sin (2 x)$ on $[0,\frac{\pi}{2}]$, we see that
\begin{align*}
\sin(2u^+)= \sin (2(\psi_{\al}+be^{-\frac{s}{2}}V) ) 
&\ge 
\sin (2\psi_\alpha) + 2be^{-\frac{s}{2}}\cos(2\psi_{\alpha_0})V,
\end{align*}
which implies
\begin{align*}
&(u^+)_s 
- (u^+)_{\rho \rho}
-\left(\frac{d-1}\rho+\frac{\rho}{2} \right)(u^+)_{\rho}
+
\frac{d-1}{2\rho^2} \sin (2u^+) 
\\
\ge &
(\psi_\al)_s 
- (\psi_\al)_{\rho \rho}
-\left(\frac{d-1}\rho+\frac{\rho}{2} \right)(\psi_\al)_{\rho}
+
\frac{d-1}{2\rho^2} \sin (2\psi_\al)
\\
&-
b e^{-\frac{s}{2}}\left\{V_{\rho\rho}+\left(\frac{d-1}\rho+\frac{\rho}2\right)
V_{\rho}-\frac{d-1}{\rho^2} \cos (2\psi_{\alpha_0})V+\frac{V}{2}\right\}
\\
=&0.
\end{align*}}
Hence $u^+$ is a supersolution of \eqref{u} as desired.
\end{proof}

\medskip

\noindent
{\it Proof of Theorem \ref{theorem:asymptotic}:} Choose first $\alpha$ such that $\psi_\alpha = \psi_N[\ell]$.
Taking $\tilde{t}(C_1, \delta)>0$ sufficiently small, 
we have  
$h_0(\sqrt{\tilde{t}}\rho) \le \psi_{\alpha_0}(\rho)$,
and then Lemma \ref{lemma:comparison} shows that
the solution of \eqref{u} globally exists and  
$$
0<u(s,\rho) \le \psi_{\alpha_0}(\rho) \qquad (s>\log \tilde{t}).
$$ 
{Let $V$, $W$, $u^+$ and $u^-$ be the functions given in Lemma \ref{lemma:u^+}.
 Lemma \ref{lemma:1} shows $\rho V$ and
 $\rho W$ have some positive limits} as $\rho \rightarrow \infty$ respectively.
Then by \eqref{snake} we can choose $b=b(\delta,C_1,C_2,R)>0$ such that
$$
u^{-}(0,\rho) 
\le h_0(\sqrt{\tilde t} \rho) = u(\log \tilde{t},\rho) 
\le u^+(0,\rho)
\qquad (\rho >0),
$$
and  Lemma \ref{lemma:comparison} yields 
$$u^{-}(s-\log \tilde{t}, \rho) 
\le 
u(s,\rho) 
\le 
u^+(s-\log \tilde{t}, \rho)
\qquad (s>\log \tilde{t}, \ \rho >0).
$$
By the definitions of $u^+$ and $u^{-}$, we have
$$
\|u(s,\rho)-\psi_N(\rho)\|_{L^\infty} \le \tilde{C}e^{-\frac{s-\log \tilde{t}}{2}}
\qquad s>\log \tilde{t}
$$
for $\tilde C = b \| W \|_\infty$. Changing back to the original variables implies  
\begin{align}
\displaystyle{\left\| h(t) -
\psi_{N} \left(\frac{\cdot}{\sqrt{t+\tilde{t}}}\right)
\right\|_{L^\infty}
\le 
\tilde{C}\tilde{t}^{\frac12}(t+\tilde{t})^{-\frac12}} \qquad {\rm for} \ t>0.
\end{align}
\boxend

\section{Uniqueness in $L^\infty$}

\label{sectionuniqueness}
As seen in the previous section, when $3 \le d\le 6$, 
there are at least two expanders evolving from the same data, 
which are both stable.  As stated in Theorem \ref{theounique}, 
we introduce an additional condition for the solution near the origin to restore uniqueness.

On the other hand, for $d \ge 7$, we will show unconditional uniqueness in $L^\infty$.
The precise statement is as  follows:

%The statement of our uniqueness theorem relies on the following %hypothesis, which we are not able to prove, even though we believe it is 
%true
%\begin{equation}
%\label{H}
%\tag{H}
%\begin{split}
%& \mbox{For any $\ell \in [0,\frac \pi 2)$, there is a unique solution $h$ %to the initial value problem} \\
%&\qquad \left\{ \begin{array}{l} h_t -h_{rr}-\frac{d-1}r h_r+\frac{d-1}{2r^2} %\sin (2h)=0 \\ h(t=0) \equiv \ell. \end{array} \right. \\
%& \mbox{such that $h \in L^\infty_t L^\infty_x$ and $0 \leq h(t,r) < %\frac{\pi}{2}$ for all $t,x$.}
%\end{split}
%\end{equation}
%It actually follows from (i) and (ii) in Theorem~\ref{theoremuniqueness} that %(H) holds if $d \geq 7$ or $3 \leq d \leq 6$ and $0 \leq \ell \leq %\frac{\pi}%{4}$. Furthermore, by Lemma~\ref{bittern} and %Corollary~\ref{egret}, we know that there exists a unique self-similar %solution satisfying the above requirements.

\begin{theorem}\label{theoremuniqueness}
\begin{itemize}
\item[(i)] If $d \geq 7$, then for any $h_0 \in L^\infty$,  there is at most one solution to~\eqref{1} in $L^\infty_t L^\infty_x$.
\item[(ii)] If $3 \leq d \leq 6$, then for any continuous $h_0$ such that $0 \leq h_0(r) \leq \frac{\pi}{4}$ for all $r\geq 0$, there exists at most one solution to~\eqref{1} such that $0 \leq h(t,r) \leq \frac{\pi}{4}$ for all $t,r \geq 0$.
\item[(iii)] If $3 \leq d \leq 6$, assume that 
$h_0$ is continuous, such that $0 \leq h_0(r) \leq \pi$ and
$$
h_0(0) < \frac{\pi}{2}, \quad |\partial_r h_0 (r)| \lesssim \frac{1}{r}, \quad |\partial_r^2 h_0 (r)| \lesssim \frac{1}{r^2}.
$$
Then for any $\delta>0$, there is at most one solution to~\eqref{1} in $L^\infty_t ([0,\delta], L^\infty_x)$ satisfying the local energy inequality
and such that
$$
h(t,r) \in [0,\frac{\pi}{2} - \delta] \quad \mbox{for $r,t \in [0,\delta]$}.
$$
Furthermore, it enjoys the bounds
$$
|\partial_r h(t,r) | \lesssim \frac{1}{\sqrt t + r}, \quad |\partial_t h(t,r)| \lesssim \frac{1}{t + r^2}.
$$
\end{itemize}
\end{theorem}

\subsection{Proof of (i)}

\noindent \underline{Step 1: the energy estimate.} Let us for the moment assume in addition to the hypotheses of the theorem that $h \in L^\infty_t L^2_x$. Consider $g$ and $h$ two solutions of~\eqref{1} with the same data $h_0 \in L^\infty \cap L^2$, and denote $H = h-g$. It satisfies the equation
$$
\partial_t H - \Delta H = F, \qquad \mbox{with} \;\; F= \frac{d-1}{2} \frac{\sin(2g) - \sin(2h)}{r^2},
$$
in the sense of distributions; notice that $F \in L^\infty_t L^1_x + L^\infty_t L^2_x$. In order to perform an $L^2$ estimate on $F$, we first mollify this equation by convolving it with $\chi^\epsilon = \frac{1}{\epsilon^d} \chi \left( \frac{\cdot}{\epsilon} \right)$, where $\chi$ is smooth, nonnegative, compactly supported, and such that $\int \chi = 1$. Denoting $f^\epsilon = \chi^\epsilon * f$, the above becomes
$$
\partial_t H^\epsilon - \Delta H^\epsilon = F^\epsilon.
$$
We can now test this inequality against $\theta(t) H^\epsilon$, with $\theta \in \mathcal{C}^\infty_0$, to obtain (keeping in mind that $H(0) = 0$)
$$
- \frac{1}{2} \int_0^\infty \theta' \| H^\epsilon \|_{L^2}^2 \,dt + \int_0^\infty \theta \| \nabla H^\epsilon \|_{L^2}^2\,ds  = \int_0^\infty \theta \langle F^\epsilon , H^\epsilon \rangle \,ds.
$$
Since on the one hand $H^\epsilon$ is bounded in $L^\infty_t L^2_x$ and in $L^\infty_t L^\infty_x$  uniformly (in $\epsilon$), and on the other hand $F^\epsilon$ is bounded in $L^\infty_t L^1_x + L^\infty_t L^2_x$ uniformly, we deduce that $\nabla H^\epsilon$ is uniformly bounded in $L^2_{t} L^2_x$. This implies that $\nabla H \in L^2_t L^2_x$. Letting $\epsilon \to 0$, we obtain 
$$
- \frac{1}{2} \int_0^\infty \theta' \| H \|_{L^2}^2 \,dt + \int_0^\infty \theta \| \nabla H \|_{L^2}^2\,ds  = \int_0^\infty \theta \langle F , H \rangle \,ds.
$$
Letting $\theta$ approach a step function, and using that almost every point is a Lebesgue point for $t \mapsto H(t)$, we obtain that
$$
\mbox{for a.e. $t$}, \quad \frac{1}{2} \| H(t) \|_{L^2}^2 \leq - \int_0^t \| \nabla H \|_{L^2}^2\,ds + \int_0^t  \langle F , H \rangle \,ds.
$$

\bigskip

\noindent \underline{Step 2: conclusion for $h_0 \in L^2$.} We now rely on Hardy's inequality with the optimal constant, which reads, for a function $f \in \dot H^1(\mathbb{R}^d)$
$$
\left\| \frac{f}{|x|} \right\|_{L^2}^2 \leq \frac{4}{(d-2)^2} \| \nabla f \|_{L^2}^2
$$
(see for instance~\cite{BCD}). Together with the energy estimate above, this gives
\begin{align*}
\frac{1}{2} \| H(t) \|_{L^2}^2 & = - \int_0^t \| \nabla H \|_{L^2}^2 \,ds + \frac{d-1}{2} \int_0^t \int_{\mathbb{R}^d} \frac{\sin(2g) - \sin(2h)}{r^2} H \,dx\,ds \\
& \leq - \int_0^t \| \nabla H \|_{L^2}^2 \,ds + (d-1) \int_0^t \int_{\mathbb{R}^d} \frac{H^2}{r^2} \,dx \,ds  \\
& \leq \left[ - 1 + 4 \frac{(d-1)}{(d-2)^2} \right]\int_0^t \| \nabla H \|_{L^2}^2 \,ds \leq 0,
\end{align*}
where the last inequality follows when $d \geq 7$, since this implies that $4 \frac{(d-1)}{(d-2)^2} < 1$.

\bigskip

\noindent \underline{Step 3: conclusion for $h_0 \in L^\infty$.} We keep the above notations for $h$, $g$, and $H$, and define for $\epsilon > 0$ 
$$
z(t,x) = z^\epsilon(t,x) = H - \epsilon \langle x \rangle.
$$
The equation satisfied by $z$ reads
\begin{align*}
(\partial_t - \Delta) z & = (\partial_t - \Delta) H + \epsilon \Delta \langle x \rangle \\
& \leq \frac{d-1}{r^2} |H| + \epsilon \Delta \langle x \rangle \\
& \leq \frac{d-1}{r^2} |z| + \epsilon \frac{C}{r}.
\end{align*}
Setting $z_+^\epsilon = z_+ = \max(z,0)$, notice that it satisfies $\operatorname{Supp} z_+ \subset B(0,\frac{R}{\epsilon})$, where $R = {2 \| H \|_{\infty}}$. An energy estimate  gives
\begin{align*}
\frac{1}{2} \frac{d}{dt} \| z_+ \|_{L^2}^2 + \| \nabla z_+ \|_{L^2}^2 & \leq (d-1) \int_{\mathbb{R}^d} \frac{(z_+)^2}{r^2} \,dx + C \epsilon \int_{\mathbb{R}^d} \frac{z_+}{r} \,dx.
\end{align*}
Using successively the optimal Hardy inequality, the support property of $z_+$, and the Cauchy-Schwarz inequality,
\begin{align*}
\frac{1}{2} \frac{d}{dt} \| z_+ \|_{L^2}^2 & \leq C \epsilon \int_{\mathbb{R}^d} \frac{z_+}{r} \,dx 
\leq C \epsilon \| z_+ \|_{L^2} \left[ \int_{B(0,\frac R \epsilon)} \frac{1}{r^2} \,dx \right]^{1/2} \\
& \leq C \sqrt{\epsilon} + C \sqrt \epsilon \| z_+ \|_{L^2}^2.
\end{align*}
This differential inequality implies that
$$
\| z^\epsilon_+ \|_{L^2}^2 \leq e^{C \sqrt \epsilon t} - 1.
$$
As $\epsilon \to 0$, we obtain $\max(H,0) = 0$. Proceeding similarly gives $\min(H,0) = 0$, and the desired result.

\subsection{Proof of (ii)}

This proof being very similar to (i), we only indicate how to treat the case where $h_0 \in L^2$ and $0\leq h_0 \leq \frac{\pi}{4}$. 
Then, considering as above $g$ and $h$ two solutions of~\eqref{1}, and denoting $H = h-g$, it satisfies
$$
H_t -H_{rr}-\frac{d-1}r H_r = \frac{d-1}{2r^2} (\sin (2g) - \sin(2h))= - \frac{d-1}{r^2} \frac{\cos(g+h) \sin (g-h)}{g-h} H
$$ 
If $g$ and $h$ are valued in $[0,\frac{\pi}{4}]$, it is easy to check that $\frac{\cos(g+h) \sin (g-h)}{g-h} \geq 0$, hence an energy estimate gives
$$
\frac{1}{2} \frac{d}{dt} \| H \|_{L^2}^2 + \| \nabla H \|_{L^2}^2 \leq 0,
$$
from which uniqueness follows.

\subsection{Proof of (iii)}

\noindent \underline{Step 1: Parabolic regularity.} By Proposition~\ref{parabolicreg2}, any solution $h$ satisfying the hypotheses of (iii) is associated to a smooth $u$, and enjoys the following bounds: for $0<t<\delta$,
$$
|\partial_r h(t,r) | \lesssim \frac{1}{\sqrt t + r}, \quad |\partial_t h(t,r)| \lesssim \frac{1}{t + r^2}
$$
as well as
$$
h(t,0) = 0 \qquad \mbox{for $t>0$}.
$$

\bigskip

\noindent \underline{Step 2: The case $h_0$ constant.} We start by proving uniqueness of the solution when $h_0$ is a constant in $[0,\frac{\pi}{2})$, and we assume that $0<h(t,r)<\frac{\pi}{2}-\delta$ for all $t,r>0$. By Corollary~\ref{egret}, there exists a unique self-similar profile $\psi$ such that $\psi(\rho) \in [0,\frac{\pi}{2})$ for all $\rho$ and $\psi(\infty) = h_0$. We want to show that $h(t,r) = \psi\left( \frac{r}{\sqrt t} \right)$.

By Step 1,
\begin{align}
\label{est:h1}
& |h(t,r)| \leq \int_0^r |\partial_r h(t,r')|\,dr' \lec \frac{r}{\sqrt{t}},
\\
& |h(t,r) - h_0| \leq \int_0^t |\partial_t h(t',r)|\,dt' \lesssim \frac{t}{r^2}.
\label{est:h2}
\end{align}
We now fix $T>0$, and let $g_0(r) = h(T,\sqrt{T}r)$, and $g$ be 
the solution of~\eqref{1} with data $g_0$. 
Note that since $g_0$ is smooth, 
the local well-posedness theory in  \cite[Theorem 1.3]{Wang} 
 shows there exist $T_1>0$ and $g(t,r)$ such that 
 $u=v[g] \in X_{T_1}$ is a solution of \eqref{0}
for the initial data $u(t=0)=v[g_0]$ 
where $v[g]$ is the corotational map defined by 
$
v[g] = (
\cos(g(|x|)), \sin(g(|x|)) \frac{x}{|x|}
)^{T}
$ and $X_T$
is the space endowed with the norm
\begin{equation*}
\| g\| _{X_T}\! :=\!
\sup_{0 < t < T}t^{\frac12}\|\nabla u(t)\|_{L^\infty}+\!\!
\sup _{x\in \R^d, \ 0 < R^2 < T}\biggl
(\frac{1}{|B(x,R)|}\int _{B(x,R)} \int^{R^2}_0
\!\!\!\!\!|\nabla u(t,y)|^2dtdy\biggl)^\frac12.
\end{equation*} 
Moreover it is shown that such a solution is unique in the class 
 $\|u\|_{X_{T_1}}<\ve_0$ with some small $\ve_0>0$.
On the other hand, from Step 1 we see 
$\tilde{g}(t,r):=h(Tt+T,\sqrt{T}r)$ 
also satisfies $\|v[\tilde{g}]\|_{X_{T_1}}<\ve_0$. 
Therefore by the uniqueness we have
\begin{equation*}
g(t,r) = h(Tt + T,\sqrt{T}r).
\end{equation*}
We now show that $g$ globally exists and satisfies
\begin{equation}
\left| g(t,r) - \psi \left( \frac{r}{\sqrt t} \right) \right| \lesssim \frac{1}{\sqrt t} \qquad {\rm for} \quad t>0 
\label{bunting}
\end{equation}
uniformly in $T>0$.
Taking Theorem \ref{theorem:asymptotic} into account, it suffices to check that, for $g_0(r):=h(T, \sqrt{T}r)$, and uniformly in $T>0$,
\begin{align}
& g_0(r) \lec r, 
\label{est:g00}
\\
& |g_0(r) - h_0|  \lesssim \frac{1}{r} \qquad \mbox{for $r>1$},
\label{est:g01}
\\
& g_0(r)\le \frac{\pi}{2}-\delta
\qquad \textrm{for} \ r>0.
\label{est:g02}
\end{align}
Indeed \eqref{est:g00} and \eqref{est:g01} 
follow from \eqref{est:h1} and \eqref{est:h2} respectively, while \eqref{est:g02} is assumed.
Therefore \eqref{bunting}  gives
$$
\left| h( Tt + T,\sqrt T r) - \psi \left( \frac{r}{\sqrt t} \right) \right| \lesssim \frac{1}{\sqrt t},
$$
or, changing variables, for $t>0$,
$$
\left| h( t + T,r) - \psi \left( \frac{r}{\sqrt t} \right) \right| \lesssim \sqrt{\frac{T}{t}}.
$$
Letting $T \to 0$ gives $h( t,r) = \psi \left( \frac{r}{\sqrt t} \right)$, which is the desired result.

\bigskip
\noindent \underline{Step 3: Stability in $L^\infty$.}
Consider $g$ and $h$ two solutions of~\eqref{1} associated to the data $h_0$ and satisfying the assumptions in (iii). We know by Step 1 that $g$ and $h$ are smooth on $(0,\delta)$. Denoting $H = h-g$, this implies, by Proposition~\ref{goldfinch}, that, for $0<t<T<\delta$,
$$
\| H(T) \|_{L^\infty_r} \leq C \| H(t) \|_{L^\infty_r}.
$$
Thus the desired equality $g = h$ would follow from 
$$
\liminf_{t \to 0} \| H(t) \|_{L^\infty_r} = 0.
$$
In order to show that this indeed holds, we fix $\eta > 0$, and we will prove that, for a constant $A>0$ yet to be determined, we can ensure that, 
for a sequence $t_n \to 0$, $\| H(t_n) \|_{L^\infty_r(\{r > A \sqrt{t_n} \})} < \eta$ and $\| H(t_n) \|_{L^\infty_r(\{r < A \sqrt{t_n} \})} < \eta$.

\bigskip

\noindent \underline{Step 4: Region $r > A \sqrt{t}$.} Writing the equation satisfied by $H$ in Duhamel form, one sees that
$$
|H(t,r)| \lesssim \int_0^t e^{(t-s)\Delta} \frac{|\sin(2h) - \sin(2g)|}{|x|^2}\,dx \lesssim \int_0^t  e^{(t-s)\Delta} \frac{1}{|x|^2}\,ds.
$$
Observing now that, for $0<\sigma < 1$, $e^{\sigma \Delta} \frac{1}{|x|^2} \leq \frac{C}{|x|^2}$, it follows that
$$
|H(t,r)| \lesssim \int_0^t \frac{1}{r^2}\,ds = \frac{t}{r^2} < \frac{1}{A^2}
$$
since $r > A \sqrt{t}$. Choosing $A$ sufficiently big, we obtain $|H(t,r)| < \eta$ for $r > A \sqrt{t}$.

\bigskip

\noindent \underline{Step 5: Region $r < A \sqrt{t}$.} Define $h^n (t,r) = h(\frac{1}{n^2} t,\frac{1}{n} r)$, and similarly $g^n$. By the parabolic regularity estimates of Step 1, $h^n$ enjoys uniform estimates (in $n$, $t$, and $r$):
\begin{align}
|h^n(t,r) | \lesssim 1, \quad |\partial_r h^n(t,r) | \lesssim \frac{1}{\sqrt{t}+r}, \quad  |\partial_t h^n(t,r) | \lesssim \frac{1}{t+r^2}
\label{asm:h}
\end{align}
as long as $t < n^2 \delta$,  and solves~\eqref{1} with data $h_0^n(r) = h_0(\frac{1}{n} r)$. 

Defining $\overline{h}(t,r)$ to be the unique (by Step 2) solution of~\eqref{1} with data $h_0(0)$, we claim that $h^n$ converges to $\overline h$ in $L^\infty ([\epsilon,R] \times B(0,R))$ for all $\epsilon>0$, $R>0$ (and, of course, so does $g^n$). If this is indeed the case, we obtain, for $t_n = \frac{1}{n^2}$, that
\begin{align*}
&\| H(t_n) \|_{L^\infty_r(\{r < A \sqrt{t_n} \})}  = \sup_{r < \frac{A}{n}} \left| H \left( \frac{1}{n^2},r \right) \right| = \sup_{r < \frac A n} |(h^n-g^n)(1,nr)| \\
& \qquad \qquad  = \| (h^n-g^n)(t=1) \|_{L^\infty(B(0,A))} \\
& \qquad \qquad\leq  \| (h^n-\overline{h})(t=1) \|_{L^\infty(B(0,A))} + \| (g^n-\overline{h})(t=1) \|_{L^\infty(B(0,A))} \\
&  \qquad \qquad \overset{n \to \infty}{\longrightarrow} 0.
\end{align*}
Therefore, for $n$ sufficiently large, $\| H(t_n) \|_{L^\infty_r(\{r < A \sqrt{t_n} \})} < \eta$, which is the desired estimate.

\bigskip

\noindent \underline{Step 6: Proof of the claim that $h^n \to \overline h$.} 
There remains to show that $h^n$ converges to $\overline h$ in $L^\infty ([\epsilon,R] \times B(0,R))$ for all $\epsilon>0$, $R>0$. 

Arguing by contradiction, assume that this is not the case. Then we can find $\widetilde \epsilon$, $\widetilde R$, $\widetilde \delta$, and a subsequence, still denoted $(h^n)$, 
such that $\| h^n - \overline h \|_{L^\infty ([\widetilde \epsilon,\widetilde R] \times B(0,R))}  > \widetilde \delta > 0$. 
As we saw, this sequence is bounded in $\operatorname{Lip}([\epsilon,R] \times B(0,R))$ for all $\epsilon, R$, thus, by the Arzela-Ascoli theorem, 
it is precompact in $L^\infty([\epsilon,R] \times B(0,R))$. 
By a diagonal argument, we can find a subsequence, still denoted $h^n$, which is convergent in $L^\infty([\epsilon,R] \times B(0,R))$ for all $\epsilon, R$; denote $\widetilde h$ the limit. 
By continuity of $h_0$, it is clear that $h^n_0$ converges to $h_0(0)$ in $L^\infty(B(0,R))$ for all $R$. 

It is now easy to pass to the limit in the weak formulation of~\eqref{1} to obtain that $\widetilde h$ is a solution of~\eqref{1} with data $h_0(0)$. Furthermore, $\widetilde h(t,r)<\frac{\pi}{2}-\delta$ for any $(t,r)$. By Step 2, 
we conclude that $\widetilde h = \overline{h}$, which is the desired contradiction.

\section{Solutions obtained by Ginzburg-Landau regularization}

\label{sectionGL}

Consider the Ginzburg-Landau penalization problem
\begin{align}\label{penalizedequ}
\left\{
\begin{array}{l}
\displaystyle \partial_t u^{\varepsilon}-\Delta u^{\varepsilon} +\frac{1}{\varepsilon^2}(|u^{\varepsilon}|^2-1)u^{\varepsilon}=0 \\
u^{\varepsilon}(t=0) = u_0
\end{array}
\right.
\end{align}
It is the gradient flow for the energy functional $\int |\nabla u|^2\,dx + \int \frac{1}{2\varepsilon^2} (1-|u|^2)^2\,dx$.

\begin{theorem}
\label{theoGL}
Assume that $u_0$ is of finite energy.
\begin{itemize}
\item[(i)] There exists a sequence $\epsilon_n \to 0$ such that $u^{\epsilon_n}$ converges to a solution $u$ of~\eqref{0} as $n \to \infty$.
\end{itemize}
Consider any solution $u$ of~\eqref{0} obtained as the limit of a subsequence $u^{\epsilon_n}$, with $\epsilon_n \to 0$.
\begin{itemize}
\item[(ii)] Then $\partial_t u \in L^2_{t,x}$. 
\item[(iii)] Under the corotional ansatz~\eqref{corotationalansatz} for the data $u_0$, let $h_0$ be associated to $u_0$. Then $u$ also enjoys the corotational ansatz~\eqref{corotationalansatz}; let $h$ be the corresponding coordinate associated to $u$. Finally assume that $h_0$ is continuous, with $h_0(0) < \frac{\pi}{2}$. Then there exists $\delta > 0$ such that $h(t,x) < \frac{\pi}{2}$ if
$|t| + |x| < \delta$.
\end{itemize}
\end{theorem}

\begin{proof}
The assertion $(i)$ can be found in~\cite{Chen} or~\cite{RSK}. The assertion $(ii)$ follows from the simple observation that $\partial_t u^\epsilon$ is uniformly bounded in $L^2_{t,x}$. 

If $u_0$ is given by the corotational ansatz~\eqref{corotationalansatz}, then so is $u$. Indeed, $u^\epsilon$ (which is unique) is of the form $u^\epsilon(t,x) = \left( \begin{array}{c} v^\epsilon(t,r) \\ w^\epsilon(t,r) \frac{x}{|x|} \end{array}\right)$, where $r = |x|$, and $v^\epsilon$ and $w^\epsilon$ are real valued functions. As $\epsilon \to 0$, $u(t,x)$ is therefore of the form $u(t,x) = \left( \begin{array}{c} v(t,r) \\ w(t,r) \frac{x}{|x|} \end{array}\right)$. Since we know furthermore that $u$ takes valued in $\mathbb{S}^d$, this implies that $u$ satisfies the corotional ansatz~\eqref{corotationalansatz}.

To prove$(iii)$, let us focus for simplicity on the case where $h_0(r)<\frac{\pi}{2}-\delta$; a localization argument easily gives the full result.
the point is that the equation satisfied by $u^\epsilon$ is subcritical in $L^\infty$, leading to existence and uniqueness of a solution in $L^\infty_t L^\infty_x$. Since the initial data takes values in the unit sphere, the solution $u$ takes values in the unit ball. Focus then on the equation satisfied by the first coordinate of $u^\epsilon$, $(u^\epsilon)^1$:
$$
\partial_t (u^\epsilon)^1 - \Delta (u^\epsilon)^1 {+} \frac{1}{\epsilon^2}(|u^\epsilon|^2 - 1)   (u^\epsilon)^1 = 0.
$$
In other words, $(u^\epsilon)^1$ solves a heat equation with a bounded potential, thus it satisfies the maximum principle, which implies that $(u^\epsilon)^1(t,x) \geq \kappa > 0$ for all $(t,x)$, and for $\kappa=\cos(\frac{\pi}2 - \delta)$. Passing to the limit leads to $u^1(t,x) \geq \kappa > 0$ for all $(t,x)$, which translates into $h(t,r) < \frac{\pi}{2} - \delta$, which is the desired result.
\end{proof}

It is easy to deduce from results we already proved uniqueness of the limit of $u^\epsilon$, as $\epsilon \to 0$.

\begin{theorem}
Assume that $d = 5$ or $6$, that the ansatz~\eqref{corotationalansatz} holds, and that $h_0$ is continuous, of finite energy, with 
$$
h_0(0) < \frac{\pi}{2}, \quad |\partial_r h_0(r)| \lesssim \frac{1}{r}, \quad \mbox{and} \quad  |\partial_r^2 h_0(r)| \lesssim \frac{1}{r^2}.
$$
Let $u$ be a limit as $\epsilon \to 0$ of a subsequence of solutions $u^\epsilon$ to~\eqref{penalizedequ}. Then $u$ is equal to the unique solution given by Theorem~\ref{theounique}.
\end{theorem}

\begin{proof} Let $u$ be as in the statement of the theorem. By Theorem~\ref{theoGL}, it satisfies $\partial_t u \in L^2_{t,x}$, thus by Proposition~\ref{energyinequality} it satisfies the local energy inequality. Furthermore, by Theorem~\ref{theoGL}, there exists $\delta>0$ such that $h(t,x) < \frac{\pi}{2}$ if $|t| + |x| < \delta$. This gives uniqueness of $u$ by Theorem~\ref{theounique}.
\end{proof}

\appendix

\section{Some ODE results}

\subsection{The comparison principle}

We give below a version of the Sturm comparison principle for which we were not able to find an appropriate reference.

\begin{lemma}
\label{comparison}
Let $p, q_1, q_2$ be $C^1$ functions with $p>0$ and $q_2 \geq q_1$ for $t>0$. Consider the differential equations:
\begin{align*}
&(p(t)x_1')'+q_1(t)x_1= 0,
\\
&(p(t)x_2')'+q_2(t)x_2= 0
\end{align*}
with initial conditions 
$$
x_1(0)=x_2(0), 
\qquad 
x'_1(0) \geq x'_2(0).
$$
If $x_2(t)>0$ on $(0,T)$, then $x_1(t)\ge x_2(t)$ on $(0,T)$.
\end{lemma}
\begin{proof}
From the equations, we have
$$
(px'_1)'x_2-(px'_2)'x_1=(q_2-q_1)x_1x_2.
$$
Integrating on $(0,t)$ implies, as long as $x_1 \geq 0$ and $x_2 \geq 0$, that
$$
[p(x'_1x_2-x_1x'_2)]^t_0 =\int^t_0 (q_2-q_1)x_1x_2
\ge 0
$$
where the last inequality holds because $q_2\ge q_1$. Since $p(x'_1x_2-x_1x'_2) \geq 0$ at $t=0$, we find
$$
x'_1x_2-x_1x'_2 \ge {0}.
$$
Therefore as long as $x_1 \geq 0$ and $x_2> 0$,
$$
\frac{d}{dt}\left( \frac{x_1}{x_2}\right)
=\frac{1}{x_2^2}
\left\{ 
x'_1x_2-x_1x'_2
\right\}
\ge 
0.
$$
Since $(x_1/x_2)(0)=1$, we conclude $x_1(t) \ge x_2(t)$ as long as $x_1 \geq 0$ and $x_2> 0$. This implies $x_1(t) >0$ on $(0,T)$.
\end{proof}

\subsection{Asymptotic behavior for a class of ODE}

We derive below a few properties of a class of ODE which plays a central role in Section~\ref{oriole}. 
Though we were not able to find a convenient reference for the result below, it is close in spirit to material in~\cite{CL}.

\begin{lemma}
\label{lemmaODE}
Consider the ODE
\begin{equation}
\label{ODEODE}
\phi''+\left(\frac{d-1}{\rho}+\frac{\rho}{2}\right) \phi' - V(\rho) \phi = 0,
\end{equation}
where $\rho$ ranges in $(0,\infty)$ and $V$ is such that $|V(\rho)| \leq \frac{C_0}{\rho^2}$ and $|V'(\rho)| \leq \frac{C_0}{\rho^3}$.

There exists a basis of solutions $\phi_1$, $\phi_2$ such that, as $\rho \to \infty$,
$$
\left\{ \begin{array}{l} \phi_1 (\rho) = \rho^{-d} e^{-\frac{\rho^2}{4}} (1 + O(1/\rho^2)) \\
\phi_1'(\rho) = -\frac{1}{2} \rho^{1-d} e^{-\frac{\rho^2}{4}} (1 + O(1/\rho^2)) \end{array} \right. \qquad
\left\{ \begin{array}{l} \phi_2(\rho) = 1 + O(1/\rho^2) \\ \phi_2'(\rho) = O(1/\rho^3). \end{array} \right.
$$
(where the implicit constants only depend on $C_0$). In particular, $\phi(\infty) = \lim_{\rho \to \infty} \phi(\rho)$ exists. 

Furthermore, there exist $R$ and $C$ depending only on $C_0$ such that, for any solution $\phi$ and $r>R$,
\begin{equation}
\label{bdcv}
\left| \phi(\rho) - \phi(\infty) \right| \leq C \frac{ |\phi(R)| + |\phi'(R)|}{\rho^2}.
\end{equation}
\end{lemma}

\begin{proof} \underline{Step 1: Change of independent variable.} 
The equation~\eqref{ODEODE} becomes, when written for $\omega(s) = \phi(\sqrt{s})$,
$$
\omega''(s) + \omega'(s) \left(\frac{d}{2s} + \frac{1}{4} \right) - W(s) \omega(s) = 0,
$$
where $W(s) = \frac{1}{4s}V(\sqrt{s})$ satisfies $|W(s)| \leq \frac{C_0}{s^2}$ and $|W'(s)| \leq \frac{C_0}{s^3}$.

\bigskip

\noindent \underline{Step 2: Diagonalization.} The ODE for $\omega$ can be written in matrix form
$$
\frac{d}{ds} \left( \begin{array}{c} \omega \\ \omega' \end{array} \right) = 
\left(  \begin{array}{cc} 0 &  1 \\ W(s) & -\frac{d}{2s} - \frac{1}{4} \end{array} \right)
\left( \begin{array}{c} \omega \\ \omega' \end{array} \right).
$$
The eigenvalues of the above matrix are given by
$$
\lambda_{\pm} = \frac{1}{2} \left( - \frac{d}{2s} - \frac{1}{4} \pm \sqrt{ \left( \frac{d}{2s} + \frac 14 \right)^2 + 4 W } \right)
$$
and enjoy the asymptotics as $s \to \infty$
$$
\lambda_- = -\frac{1}{4} - \frac{d}{2s} + O \left( \frac{1}{s^2} \right), \quad \lambda_+ = 4 W(s) + O \left( \frac{1}{s^3} \right).
$$
In particular, they are real and different for $s$ sufficiently big, which we assume in the following.
A basis of eigenvectors is provided by $e_{\pm} = \left( \begin{array}{c} 1 \\ \lambda_{\pm} \end{array} \right)$, with dual basis $\widetilde{e}_{\pm} = \frac{1}{\lambda_{\pm} - \lambda_{\mp}} \left( \begin{array}{c} -\lambda_{\mp} \\ 1 \end{array} \right)$ such that $\langle e_{\pm} \,,\, \widetilde{e}_{\pm} \rangle = 1$ and $\langle e_{\pm} \,,\, \widetilde{e}_{\mp} \rangle = 0$. Decomposing $\left( \begin{array}{c} \omega \\ \omega' \end{array} \right)$ in this basis, 
$$
\left( \begin{array}{c} \omega \\ \omega' \end{array} \right) 
= \alpha_+ e_+ + \alpha_- e_-,
$$
and the ODE becomes
$$
\left\{ 
\begin{array}{l} \alpha_+' = \lambda_+ \alpha_+ + f_1 \alpha_+ + f_2 \alpha_- \\
 \alpha_-' = \lambda_- \alpha_- + f_3 \alpha_+ + f_4 \alpha_- \end{array} 
\right. ,
$$
where $f_1 = - \langle e_+'\,,\,\widetilde{e}_+ \rangle$, $f_2 = - \langle e_-'\,,\,\widetilde{e}_+ \rangle$, 
$f_3 =- \langle e_+'\,,\,\widetilde{e}_- \rangle$, and $f_4 = - \langle e_-'\,,\,\widetilde{e}_- \rangle$ all enjoy bounds of the type $|f_i(s)| < \frac{C}{s^2}$, $|f_i'(s)| < \frac{C}{s^3}$.

\bigskip

\noindent
\underline{Step 3: Uniform bound.} First choose $S$ such that $-1 <\lambda_-(s)<-\frac{1}{8}$ and $-\frac{1}{16}<\lambda_+(s)<1$ for $s>S$. Then
$\alpha_{\pm} = \langle (\omega,\omega')\,,\,\widetilde{e}_{\pm} \rangle $ and therefore
$$
|\alpha_+(S)| + |\alpha_-(S)| \leq C \left( |\omega(S)| + |\omega'(S)| \right).
$$
Furthermore, for $s>S$,
\begin{align*}
\frac{d}{ds} \left( |\alpha_+| + |\alpha_-| \right) & \leq \lambda_- |\alpha_-| + \left( |\lambda_+| + |f_1| + |f_2| + |f_3| + |f_4| \right) \left( |\alpha_+| + |\alpha_-| \right) \\
& \leq \left( |\lambda_+| + |f_1| + |f_2| + |f_3| + |f_4| \right) \left( |\alpha_+| + |\alpha_-| \right),
\end{align*}
thus by Gronwall's lemma, for any $s > S$,
\begin{align*}
\left( |\alpha_+(s)| + |\alpha_-(s)| \right) & \leq e^{ \int_{S}^s  |\lambda_+| + |f_1| + |f_2| + |f_3| + |f_4|} \left( |\alpha_+(S)| + |\alpha_-(S)| \right) \\
& \leq C \left( |\omega(S)| + |\omega'(S)| \right),
\end{align*}
which is the desired uniform bound.
\bigskip

\noindent
\underline{Step 4: Uniform convergence.} The previous uniform bound immediately gives that  
$|\alpha_+'(s)| = O \left( \frac{1}{s^2} \right)$, from which we see that 
\begin{equation}
\label{bdplus}
|\alpha_+(s) - \alpha_+(\infty)| \leq C \frac{|\omega(S)| + |\omega'(S)|}{s}.
\end{equation}
Next, the equation on $\alpha_-$ can be written
$$
\left( s^{\frac{d}{2}} e^{\frac{s}{4}} \alpha_- \right)'
= s^{\frac{d}{2}} e^{\frac{s}{4}} \left( f \alpha_- + f_3 \alpha_+ + f_4 \alpha_- \right),
$$
where $f = \lambda_- + \frac{1}{4} + \frac{d}{2s}$ is such that $|f(s)| < \frac{C}{s^2}$. 
Thus, by the uniform bounds, for $s>S$,
$$
\left| \left( s^{\frac{d}{2}} e^{\frac{s}{4}} \alpha_- \right)' \right| \leq  C s^{\frac{d}{2}-2} e^{\frac{s}{4}} (|\omega(S)| + |\omega'(S)|).
$$
Integrating this differential inequality gives, for $s>S$
\begin{equation}
\label{bdminus}
|\alpha_-(s)| \leq C\left( \frac{S^{\frac{d}{2}} e^{\frac{S}{4}}}{s^{\frac{d}{2}} e^{\frac{s}{4}}} + \frac{1}{s^{\frac{d}{2}}e^{\frac{s}{4}}}  
\int_S^s \sigma^{\frac{d}{2}-2} e^{\frac{\sigma}{4}} \,d\sigma \right) (|\omega(S)| + |\omega'(S)|) \leq C \frac{|\omega(S)| + |\omega'(S)|}{s^2}.
\end{equation}
Combining~\eqref{bdplus} and~\eqref{bdminus}, and coming back to the original independent variable, gives~\eqref{bdcv}.

\bigskip

\noindent
\underline{Step 5: Construction of $\phi_1$.} Recall that $f(s) = \lambda_-(s) + \frac{1}{4} + \frac{d}{2s}$ is such that $|f(s)| < \frac{C}{s^2}$. 
Setting furthermore $\beta_-(s) = s^{\frac{d}{2}} e^{\frac s4} \alpha_-(s)$, it satisfies the equation
$$
\left\{ \begin{array}{l} \alpha_+' = \lambda_+ \alpha_+ + f_1 \alpha_+ + f_2 \beta_- s^{-\frac{d}{2}} e^{-\frac s4}\\
\beta_-' = (f_4 + f) \beta_- + f_3 s^{\frac{d}{2}} e^{\frac s4} \alpha_+
\end{array} \right. .
$$
We are looking for a solution such that $\beta_- \to 1 $ and $\alpha_+ < C t^{-\frac{d}{2}-2} e^{-\frac t4}$ as $t \to \infty$. This is equivalent to solving the integral system
$$
\left\{ \begin{array}{l} 
\alpha_+(s) = - \int_s^\infty \left( (f_1+\lambda_+) \alpha_+ + f_2  \beta_- \sigma^{-\frac{d}{2}} e^{-\frac \sigma 4}\right)\,d\sigma \\
\beta_-(s) = 1 - \int_s^\infty \left( (f_3 + f)\beta_- + f_4 \sigma^{\frac{d}{2}} e^{\frac \sigma 4} \alpha_+\right)\,d\sigma.
\end{array} \right.
$$
It is easy to solve this integral system on $[T,\infty)$, for $T$ sufficiently big, by a contraction argument in the space given by the norm 
$\| \beta_-\|_\infty + \| s^{\frac d2 + 1} e^{\frac s4} \alpha_+ \|_\infty$. This gives a decay $\sim  s^{-\frac d2 - 1} e^{-\frac s4}$ for $\alpha_+$, but using once again the equation improves it to $s^{-\frac d2 - 2} e^{-\frac s4}$. Summarizing, we obtain
$$
|\alpha_+| \leq C s^{-\frac d2 - 2} e^{-\frac s4}, \quad \mbox{and} \quad \alpha_- = s^{-\frac d2} e^{-\frac s4} (1+ O (1/s)).
$$
Coming back to the $\phi$ variable gives the desired solution, $\phi_1$.

\bigskip

\noindent
\underline{Step 6: Construction of $\phi_2$.} Solve the ODE for $s>s_0>0$ with data $\alpha_+(s_0) = 1$, $\alpha_-(s_0) = 0$. By the argument in steps 3 and 4, we see that
$$
|\alpha_+(s_0) - \alpha_+(\infty)| \leq \frac{C}{s_0},
$$
guaranteeing that $\alpha_+(\infty)>0$ for $s_0$ sufficiently big. Furthermore, the decay of $\alpha_-$ is given by \eqref{bdminus}. Rescaling $(\alpha_+,\alpha_-)$, 
we can ensure that $\alpha_+(\infty) = 1$. Coming back to the $\phi$ variable gives the desired solution, $\phi_2$.
\end{proof}

\begin{corollary} 
\label{loon} \begin{itemize}
\item[(i)]
Given a potential $V$ satisfying the hypotheses of Lemma~\ref{lemmaODE}, there exists $R_0>0$ such that: if $\phi$ is a solution of~\eqref{ODEODE} with
$$
\phi(R) = 0,\quad \phi'(R)>0,
$$
with $R>R_0$, then $\phi(\rho)\neq 0$ for $\rho>R$, and $\phi(\infty) \neq 0$. 
\item[(ii)] Any nonzero solution of~\eqref{ODEODE} can vanish only once to the right of $R_0$.
\end{itemize}
\end{corollary}
\begin{proof} $(i)$ Consider $\phi$ as in the statement of the corollary, and decompose it in the basis given by $\phi_1$ and $\phi_2$: $\phi = \alpha \phi_1 + \beta \phi_2$. Since $\phi(R) = 0$, we obtain
$$
\alpha = - R^d e^{R^2/4} (1 + O(R^{-2})) \beta.
$$
Therefore, $\alpha$ and $\beta$ are nonzero. By linearity, we can assume without loss of generality that $\beta = 1$.
Using the expressions for $\phi_1$, $\phi_1'$, $\phi_2$, and $\phi_2'$ in Lemma~\ref{lemmaODE}, we obtain that
\begin{align*}
& \phi(\rho) = 1 + O(\rho^{-2}) - \frac{R^d e^{R^2/4}}{\rho^d e^{\rho^2/4}}(1 + O(R^{-2})) \\
& \phi'(\rho) = O(\rho^{-3}) + \frac{1}{2} \frac{R^d e^{R^2/4}}{\rho^{d-1} e^{\rho^2/4}}(1 + O(R^{-2})).
\end{align*}
It is now easy to see from these expressions that, choosing $R$ sufficiently big, $\phi'(\rho)>0$ for $R < \rho < R+\frac{100}{R}$, 
while $\phi(\rho) > \frac{1}{2}$ for $\rho > R+\frac{100}{R}$. This implies the desired result.

\bigskip
\noindent $(ii)$ It is a simple consequence of $(i)$ once one observes that, for any $\rho_0$, $\phi(\rho_0) = \phi'(\rho_0) = 0$ implies that $\phi\equiv 0$. 
\end{proof}

\begin{corollary} \label{grebe} There exists $R_0>0$ such that any solution of~\eqref{ODE} with
$$
\psi(R) = \frac{\pi}{2}, \quad \psi'(R) > 0
$$
at $R>R_0$ satisfies $\psi(\rho) > \frac{\pi}{2}$ for $\rho>R$, and $\psi(\infty)>\frac{\pi}{2}$.  In particular, any solution of~\eqref{ODE} can cross $\frac{\pi}{2}$ only once to the right of $R_0$.
\end{corollary}
\begin{proof} Setting $\phi = \psi - \frac{\pi}{2}$, observe that it solves
$$
\phi'' + \left( \frac{d-1}{\rho} + \frac{\rho}{2}  \right) \phi' - V(\rho) \phi = 0, \quad \mbox{with} \quad \left\{ \begin{array}{l} \phi(R) = 0 \\ \phi'(R)>0 \end{array} \right.
$$
where we set
$$
V(\rho) = - \frac{d-1}{2\rho^2} \frac{\sin(2 \phi(\rho))}{\phi(\rho)}.
$$
By the decay estimates in Proposition~\ref{GR}, $V$ satisfies the hypotheses of Lemma~\ref{lemmaODE}, and thus Corollary~\ref{loon} applies, giving the desired result.
\end{proof}

\section{General properties of the harmonic map heat flow}

\subsection{Parabolic regularity}

\begin{proposition}
\label{parabolicreg}
If $h_0$ satisfies
$$
|h_0(r)| \lesssim 1,\quad |\partial_r h_0(r)| \lesssim \frac{1}{r},
$$
then any bounded solution $h(t,r)$ of~\eqref{1} on $[0,T]$ 
satisfies 
$$
|\partial_r h(t,r)| \lesssim \frac{1}{r}.
$$
uniformly on $[0,T]$.
If $h_0$ further satisfies
$|\partial_r^2 h_0(r)| \lesssim \frac{1}{r^2}$, then 
$$
 |\partial_r^2 h (t,r)| \lesssim \frac{1}{r^2} \quad \mbox{and} \quad |\partial_t h(t,r)| \lesssim \frac{1}{r^2}.
$$
holds uniformly on $[0,T]$.
\end{proposition}

\begin{proof} We only deal with the estimate
$|\partial_r h(t,r)| \lesssim \frac{1}{r}$. 
We see by the Duhamel principle that
$$
h(t)=G_t*h_0-\int^t_0 G_{t-s}*\left(\frac{(d-1)\sin 2h}{2|x|^2}\right)ds,
$$
where $G_t$ is the heat kernel in $\R^d$.
It suffices to show that
$$
\int^t_0 |\nabla_x G_{t-s}| *\left(\frac{1}{|x|^2}\right)ds \lec \frac{1}{|x|},
$$
Using the heat kernel estimate 
$|\nabla_x G_s(x)| \lec (s+|x|^2)^{\frac{-d-1}{2}}$, 
we have
\begin{align*}
\left|\int^t_0\nabla_x G_{t-s}*\left(\frac{1}{|x|^2}\right)ds\right|
&\lec \int \int^t_0  (s+|y|^2)^{\frac{-d-1}{2}}ds |x-y|^{-2}dy
\\
&\lec \int [|y|^{-d+1}-(t+|y|^2)^{\frac{-d+1}{2}}] |x-y|^{-2}dy
\\
&\lec \frac{1}{|x|},
\end{align*}
which is the desired estimate. 
 \end{proof}

\bigskip

For general solutions of the harmonic map heat flow with $3 \leq d \leq 6$, the above result is optimal in that a decay estimate such as $|\partial_r h(t,r)| \lesssim \frac{1}{t^\alpha}$, for some $\alpha>0$, cannot be expected. This can be seen by taking $h_0 \equiv \frac{\pi}2$, 
for which a solution is given by $h(t,r) = h_0$ for $t<T$, and 
$h(t,r) = \psi (\frac{r}{\sqrt t-T}) $ for $t>T$, where $\psi$ is a self-similar profile associated to $h_0$.

However, this result can be improved under further assumptions: this is the content of the following two propositions.

\begin{proposition}
\label{parabolicreg3}
Assume that $h$ is a solution of ~\eqref{1} on $[0,T]$ 
with data $h_0$ satisfying
$$
|\partial_r h_0(r)|\lesssim \frac{1}{r} \quad \mbox{and} \quad |h(t,r)| \lesssim \min\left( 1, \frac r {\sqrt{t}} \right).
$$
Then we have
$$
|\partial_r h(t,r)| \lesssim \frac{1}{r + \sqrt t} \left( 1 + \langle \log \frac{r}{\sqrt t} \rangle \mathbf{1}_{r < \sqrt t} \right). 
$$
\end{proposition}
\begin{proof} Arguing as in the previous proposition, we write
$$
h(t)=G_t*h_0-\int^t_0 G_{t-s}*\left(\frac{(d-1)\sin 2h}{2|x|^2}\right)ds.
$$
The desired bound on $G_t*h_0$ is immediate due to the assumption that $|\partial_r h_0(r)|\lesssim \frac{1}{r}$. To estimate the integral term, we use the hypothesis that $|h(t,r)| \lesssim \min\left( 1, \frac r {\sqrt{t}} \right)$ to obtain
$$
\left| \int^t_0 \nabla G_{t-s}*\left(\frac{(d-1)\sin 2h}{2|x|^2}\right)ds \right| \lesssim \left| \int^t_0 \nabla G_{t-s}*\left(\frac{1}{|x|(|x| + \sqrt s)}\right)ds \right|.
$$
Using the heat kernel bound $|\nabla_x G_s(x)| \lec (s+|x|^2)^{\frac{-d-1}{2}}$, one can prove that
\begin{align*}
& \left| \nabla G_{t-s}*\left(\frac{1}{|x|(|x| + \sqrt s)}\right) \right|\lesssim \frac{1}{\sqrt{t-s}( \sqrt{t-s} + |x|)(\sqrt s + |x|)} \quad \mbox{if $s > t-s$} \\
&  \left| \nabla G_{t-s}*\left(\frac{1}{|x|(|x| + \sqrt s)}\right) \right| \lesssim \frac{1}{\sqrt{t-s}( t-s + |x|^2)}\qquad \mbox{if $s < t-s$}.
\end{align*}
One can then easily check that $\left| \int^t_0 \nabla G_{t-s}*\left(\frac{1}{|x|(|x| + \sqrt s)}\right)ds \right|$ satisfies the desired bound.
\end{proof}

\begin{proposition} 
\label{parabolicreg2} 
Assume that 
$$
|h_0(r)| \lesssim 1,\quad |\partial_r h_0(r)| \lesssim \frac{1}{r}, \quad |\partial_r^2 h_0(r)| \lesssim \frac{1}{r^2},
$$
and that $h$ solves~\eqref{1} on $[0,T]$ , satisfies the local energy inequality~\eqref{EnIn}, and, 
$$
h(t,r) \in [0,\frac{\pi}{2} - \delta]  \quad \mbox{for $(t,r) \in [0,T] \times [0,\delta]$}.
$$
with some $\delta, T>0$. 
Then $h(t,0)=0$ for $t>0$ and $h$ enjoys the bounds
$$
|\partial_r h(t,r)| \lesssim \frac{1}{\sqrt t + r},  \qquad |\partial_t h(t,r)| \lesssim \frac{1}{t + r^2}
\quad \mbox{for $(t,r) \in [0,T] \times [0,\infty)$}.
$$
\end{proposition}

\begin{proof} 
\bigskip 
\noindent
\underline{Step 1: Space-time bounds on the gradient.} 
We first treat the case where $h(t,r) \in [0,\frac{\pi}{2} - \delta]$ for any $(t,r)$.
Recall that the solution $u$ of the general formulation~\eqref{0} and the solution $h$ of the equivariant formulation~\eqref{1} of the harmonic map heat flow are related by
\eqref{corotationalansatz}
Therefore, $|\nabla u|^2 = |\partial_r h|^2 + \frac{d-1}{2r^2} \sin(h)^2$.
This formula implies that
$$
(\partial_t - \Delta) h^2 = - \frac{d-1}{r^2} \sin(2h) h - 2 |\nabla h|^2 
\lesssim -|\nabla u|^2,
$$
where the last inequality follows since $0\leq h \leq \frac{\pi}{2}-\delta$.
Duhamel's formula gives that
\begin{align}
\label{est:h^2}
h^2(t,x) - e^{t\Delta} h_0^2 = \int_0^t e^{(t-s) \Delta} \left[ - \frac{d-1}{r^2} \sin(2h) h - 2 |\nabla h|^2\right] (s)\,ds \lesssim - \int_0^t e^{(t-s) \Delta} |\nabla u|^2 (s)\,ds
\end{align}
Therefore, by boundedness of $h$ and $h_0$, for any $(t,x) \in [0,T] \times \mathbb{R}^d$, 
\begin{equation}
\label{owl}
\int_0^t \int\frac{1}{(4\pi(t-s))^{d/2}} e^{-\frac{|x-y|^2}{4(t-s)}} |\nabla u|^2 (s,y)\,dy\,ds = 
\int_0^t [e^{(t-s) \Delta} |\nabla u|^2(s)](x) \,ds  
\lesssim 1.
\end{equation}

\bigskip
\noindent
\underline{Step 2: Smoothness of $u$.} Our first goal is to prove that $u$ is smooth for $t>0$. By Proposition~\ref{parabolicreg}, $u$ is smooth away from $r=0$, so we only need to prove that $u$ is smooth in a neighbourhood of $r=0$. 

For $t>0$, define the (overlapping) parabolic cylinders
$$
P^{t}_j = \left\{ (s,x) \in \left[ t - 2^{-2j} - \frac{5}{6} 2^{-2j},  t - 2^{-2j} + \frac{5}{6} 2^{-2j} \right] \times B \left( 0, 2^{-j} \right) \right\}.
$$
Taking $x=0$, the bound~\eqref{owl} implies that
$$
\sum_{2^j > \frac{\sqrt t}{100}} 2^{-jd} \int_{P^{t}_j} |\nabla u|^2 \,dx\,ds \lesssim 1.
$$
This implies that, for every $\epsilon>0$, $\int_{P^{t}_j} |\nabla u|^2 \,dx\,ds < \epsilon$ for $j$ sufficiently large, say $j>j_0$. Pick now $\epsilon$ as in the $\epsilon$-regularity theorem appearing as  in~\cite[Proposition 5.3]{Moser}. 
As a consequence, for $j$ sufficiently big, $u$ is smooth on $Q^{t}_j$, where $Q^{t}_j$ is a parabolic cylinder with the same center as $P^{t}_j$, but slightly smaller. The union of the $Q^t_j$ for $j>j_0$ contains a (parabolic) cone with top $(t,x)$, over which $u$ is smooth. Therefore, $u$ is smooth on a time slice $[t-\alpha, t]$, with $\alpha>0$.

One can now use a quantitative argument (as will be done in Step 3 below) to obtain a bound on $(\partial_t u, \nabla u)$ on $[t-\gamma^2,t] \times B(0,\gamma)$ for some $\gamma>0$. By using local well-posedness theory, we obtain that $u$ is smooth in a neighbourhood of $(t,0)$. Since this holds for any $t>0$, $u$ is smooth for $t>0$.

As a first consequence, $h(t,0) = 0$ for $t>0$.

\bigskip
\noindent
\underline{Step 3: Quantitative bounds on $\nabla u$.} Fix $t>0$, $x \in \mathbb{R}^d$, and define
$$
\Phi(R) = R^2 \int_{\mathbb{R}^d} \frac{1}{(4 \pi R^2)^{d/2}} e^{-\frac{|x-y|^2}{4R^2}} |\nabla u|^2 (t-R^2,y)\,dy.
$$
The bound~\eqref{owl} implies that
$$
\int_0^t \frac{1}{\sigma} \Phi(\sqrt{\sigma}) \,d\sigma \lesssim 1.
$$
Therefore, there exists $\tau\ge ct$ with some $c>0$ such that $\Phi(\sqrt{\tau}) < \epsilon_0^2$. Here, $\epsilon_0$ is as in the $\epsilon$-regularity result appearing as  in~\cite[Proposition 7.1.4]{LW}, which implies that
$$
|\partial_t u(t,x)| \lesssim \frac{1}{\tau} \lesssim \frac{1}{t}, \quad |\nabla u(t,x)| \lesssim \frac{1}{\sqrt \tau} \lesssim \frac{1}{\sqrt t}.
$$
This gives the desired result in combination with Proposition~\ref{parabolicreg}.

\bigskip

\noindent
\underline{Step 4: General case.} 
The general case follows by a standard localization argument, 
Indeed by Proposition~\ref{parabolicreg}, we may show that 
$|\partial_r h(t,r)| \lec t^{-\frac12} $ for 
$(t,r) \in [0,T] \times [0,\delta]$. 
Let $\phi$ be a cut off fucntion defined in \eqref{def:cutoff}
 with $R =\frac{\delta}{2}$ and $\tilde{h}:=h\phi$, then $\tilde{h}$ satisfies $\tilde{h} \in [0,\frac{\pi}{2}-\delta]$ and 
$$
\tilde{h}_t -\tilde{h}_{rr}-\frac{d-1}r \tilde{h}_r+\frac{d-1}{2r^2} 
\sin (2\tilde{h})=F,
$$
with some compactly supported bounded function $F$. 
Therefore arguing as in the previous steps, 
we obtain the desired bound.
\end{proof}

\subsection{The local energy inequality}

The following proposition gives criteria for a solution of the harmonic map heat flow to satisfy the local energy inequality~\eqref{EnIn}.

\begin{proposition}
\label{energyinequality}
For $h_0$ locally in $H^1_x$, and $h$ a bounded solution on $[0,T) \times \mathbb{R}^d$ locally in $H^1_{t,x}$, assume that 
\begin{itemize}
\item[(i)] Either 
$h$ enjoys the bounds
$$
|h(t,r)| \lesssim \frac{r}{\sqrt t + r},\quad  |\partial_r h(t,r)| \lesssim \frac{1}{\sqrt t + r}, \quad \mbox{and} \quad |\partial_t h(t,r)| \lesssim \frac{1}{t + r^2}.
$$
\item[(ii)] Or $d \geq 5$, $\partial_t h \in L^2_{t,x}$ locally, and $h_0$ satisfies
$$
|h_0(r)| \lesssim 1,\quad |\partial_r h_0(r)| \lesssim \frac{1}{r}, \quad |\partial_r^2 h_0(r)| \lesssim \frac{1}{r^2}.
$$
\item[(iii)] Or $h$ enjoys the bounds
$$
|h(t,r)| \lesssim \frac{r}{r+\sqrt t},\quad |\partial_t h(t,r)| \lesssim \frac{1}{r^2}, \quad \mbox{and} \quad |\partial_r h(t,r)| \lesssim \frac{1}{\sqrt t + r} \left( 1 + \langle \log \frac{r}{\sqrt  t} \rangle \mathbf{1}_{r < \sqrt t} \right) 
$$
\end{itemize}
Then $h$ belongs locally to $H^1_{t,x}$, and satisfies the local energy 
inequality~\eqref{EnIn}\footnote{The proof actually shows that the energy equality (where the sign $\geq$ in~\eqref{EnIn} is replaced by $=$) is satisfied}.
\end{proposition}

\begin{proof} \underline{Case $(i)$.} We will be working with $u$ instead of $h$; recall that it is given by the ansatz~\eqref{corotationalansatz}. Since
$$
|\nabla u| \lesssim |\partial_r h| + \frac{1}{r} |h|, \quad \mbox{and} \quad |\partial_t u| \lesssim |\partial_t h|,
$$
our assumptions imply that $u$ is locally $H^1_{t,x}$. 

Let $\chi$ be a smooth, nonnegative, compactly supported function on $\mathbb{R}^d$ with integral one: $
\int_{\mathbb{R}^d}  \chi = 1$. For a function $g$ on $\mathbb{R}^d$, define its regularization $g^\epsilon = \frac{1}{\epsilon^d} \chi\left(\frac{\cdot}{\epsilon} \right) * g$. The regularization of the equation satisfied by $u$ reads
$$
\left\{
\begin{array}{l}
\partial_t u^\epsilon - \Delta u^\epsilon = ( |\nabla u|^2 u)^\epsilon \\
u^\epsilon(t=0) = u^\epsilon_0.
\end{array}
\right.
$$
The local energy inequality can be derived by taking the scalar product of the above by $\tau \partial_t u^\epsilon + \psi^\ell \partial_\ell u^\epsilon$, 
and integrating the equation on $[0,T] \times \mathbb{R}^d$. Indeed, the left-hand side gives
\begin{align*}
& \int_0^\infty \int_{\mathbb{R}^d} (\partial_t u^\epsilon - \Delta u^\epsilon) (\tau \partial_t u^\epsilon + \psi^\ell \partial_\ell u^\epsilon)\,dx\,dt \\
& = \int_0^\infty \int_{\mathbb{R}^d} \left[ \tau |\partial_t u^\epsilon|^2 - \frac{1}{2} \left( \partial_t \tau + \partial_\ell \psi^\ell \right) |\nabla u^\epsilon|^2
+ \partial_i \tau\partial_i (u^\epsilon)^k  \partial_t (u^\epsilon)^k \right.\\
& \qquad \qquad \left. + \psi^\ell \partial_t (u^\epsilon)^k  \partial_\ell (u^\epsilon)^k 
+ \partial_i \psi^\ell \partial_i (u^\epsilon)^k  \partial_\ell (u^\epsilon)^k \right] \,dx\,dt - \int_{\mathbb{R}^d} \frac{1}{2} \tau(t=0) |\nabla u_0^\epsilon|^2 \,dx, 
\end{align*}
which converges to the desired expression as $\epsilon \to 0$. Therefore, it suffices to check that the right-hand side converges to zero: we need to show that
$$
\int_0^T \int_{\mathbb{R}^d}( |\nabla u|^2 u)^\epsilon \cdot \left[ \tau \partial_t u^\epsilon + \psi^\ell \partial_\ell u^\epsilon \right] \,dx\,dt \overset{\epsilon \to 0}{\longrightarrow} 0.
$$
Since $u$ is smooth away from $(t,x) = (0,0)$, we get, for $(t,x)$ away from $(0,0)$, that $( |\nabla u|^2 u)^\epsilon \cdot \left[ \tau \partial_t u^\epsilon + \psi^\ell \partial_\ell u^\epsilon \right] \to 0$ uniformly as $\epsilon \to 0$. 
Therefore, the only problem is to deal with the singularity at $(t,x) = (0,0)$. But our assumptions imply that $|\nabla u|^2 u \cdot \left[ \tau \partial_t u + \psi^\ell \partial_\ell u\right]$ is integrable near $(t,x) = (0,0)$, thus giving the desired result.

\bigskip

\noindent
\underline{Case $(ii)$.} Observe that, by Proposition~\ref{parabolicreg}, $|\nabla u| \lesssim \frac{1}{r}$. Since $d \geq 5$, it means that $|\nabla u|^2 \in L^2_{t,x}$; since furthermore $\partial_t u \in L^2$, we obtain that $|\nabla u^\epsilon|^2 u^\epsilon \cdot \partial_t u^\epsilon \rightarrow |\nabla u|^2 u \cdot \partial_t u = 0$ in $L^1$ as $\epsilon \to 0$, hence the desired result follows as in case $(i)$.

\bigskip

\noindent
\underline{Case $(iii)$.} Once again, it suffices to observe that the assumptions made on $h$ imply that $|\nabla u|^2 u \cdot \left[ \tau \partial_t u + \psi^\ell \partial_\ell u\right]$ is integrable near $(t,x) = (0,0)$.
\end{proof}

\subsection{Stability in $L^\infty$}

The following result can be found in~\cite{JK}. We recast the proof in the framework of corotational solutions, for which it becomes very elementary.

\begin{proposition} 
\label{goldfinch}

Consider $h_1$ and $h_2$ two solutions of~\eqref{1} on $[0,T] \times (0,\infty)$ which are space and time continuous at $t=0$ and smooth for $t>0$.

\begin{itemize}
\item[(i)] For $\delta > 0$ there exists $C = C(\delta)$ such that: if $h_1(t,r), h_2(t,r) \in [0, \frac{\pi}{2} - \delta]$ for all $(t,r) \in [0,T] \times (0,\infty)$, then
$$
\| (h_1 - h_2) (t=T) \|_\infty \leq C \| (h_1 - h_2) (t=0) \|_\infty
$$
\item[(ii)] For $\delta, T > 0$ there exists $C = C(\delta,T)$ such that: if $h_1(t,r), h_2(t,r) \in [0, \frac{\pi}{2} - \delta]$ for all $(t,r) \in [0,T] \times (0,\tau)$, then
$$
\| (h_1 - h_2) (t=T) \|_\infty \leq C \| (h_1 - h_2) (t=0) \|_\infty
$$
\end{itemize}
\end{proposition}

The proof of the previous proposition essentially consists of the following lemma. Indeed, it suffices to notice that $\theta$ or $\Theta$, defined below, control $|h_1 - h_2|^2$, 
and to apply the maximum principle.

\begin{lemma} 
\begin{itemize}
\item[(i)] Under the assumptions of (i) in the previous proposition, let $\theta = \frac{1 - \cos(h_1-h_2)}{\cos h_1 \cos h_2}$ and $A = \cos(h_1) \cos(h_2)$. Then
$$
A \partial_t \theta - \nabla \cdot (A \nabla \theta) \leq 0.
$$
\item[(ii)] Under the assumptions of (ii) in the previous proposition, 
let $\chi(r)$ be a smooth, positive, nonincreasing function on $\mathbb{R}_+$ such that $\chi = 1$ if $r<\frac{\delta}{2}$, $\chi=0$ if $r>\delta$ and $A(t,r)$ be a positive, 
smooth function away from $t=0$ or $r=0$ such that
$$
A(t,r) =
\left\{ \begin{array}{ll}
-\log( \cos(h_1) \cos(h_2)) & \mbox{if $r<\delta$}\\
1 & \mbox{if $r>2\delta$}.
\end{array} \right.
$$
Define further 
$$
\Theta = \chi \theta(h_1,h_2) + (1 - \chi) |h_2 - h_1|^2 \quad \mbox{where} \quad \theta(h_1,h_2) = \frac{1 - \cos(h_1-h_2)}{\cos h_1 \cos h_2}.
$$
Then
$$
A \partial_t \Theta - \nabla \cdot (A \nabla \Theta) \lesssim \Theta + C |\nabla \Theta|.
$$
\end{itemize}
\end{lemma}

\noindent
\begin{proof}
\underline{Proof of (i)} Let us introduce some further notations: set, for $i = 1,2$, $\phi_i = \sin h_i$ and $\psi_i = \cos h_i$.
A lengthy but straightforward computation shows that
\begin{align*}
& A \partial_t \theta(h_1,h_2) - \nabla \cdot (A  \nabla \theta(h_1,h_2)) \\
& \qquad = - \underbrace{\left[  (\phi_1 \psi_2 - (1+\theta) \psi_1 \phi_2 ) (\partial_t - \Delta) h_2 + (\phi_2 \psi_1 - (1+\theta) \psi_2 \phi_1)(\partial_t - \Delta) h_1 \right] }_{I} \\
& \qquad \quad \underbrace{+\frac{1}{\psi_1 \psi_2}  (\psi_1^2 + \psi_2^2) \partial_r h_1 \partial_r h_2- (1+\theta) (\psi_2^2 (\partial_r h_1)^2 + \psi_1^2 (\partial_r h_2)^2)}_{II}.
\end{align*}
We will now show that $I \geq 0$ while $II \leq 0$. Starting with $I$, observe that
$$
[\phi_1 \psi_2 - (1+\theta) \psi_1 \phi_2 ] (\partial_t - \Delta) h_2 = -\frac{d-1}{r^2}(\phi_1 \psi_2 - (1+\theta) \psi_1 \phi_2) \psi_2 \phi_2
= - \frac{d-1}{r^2}(\phi_1 \phi_2 - \phi_2^2),
$$
with the symmetric formula 
$$
[\phi_2 \psi_1 - (1+\theta) \psi_2 \phi_1 ] (\partial_t - \Delta) h_1 = - \frac{d-1}{r^2}(\phi_1 \phi_2 - \phi_1^2).
$$
Then
$$
I = \frac{d-1}{r^2} (\phi_1 - \phi_2)^2 \geq 0.
$$
Turning to $II$, it is a quadratic form in $(\partial_r h_1, \partial_r h_2)$, which can be represented by the matrix 
$\frac{1}{\psi_1 \psi_2} \left( \begin{array}{cc}- (1+\theta) \psi_2^2 & \frac{1}{2} (\psi_1^2 + \psi_2^2) \\   \frac{1}{2} (\psi_1^2 + \psi_2^2) &- (1+\theta) \psi_1^2 \end{array} \right)$. The trace of this matrix is obviously nonpositive, so $II \leq 0$ if and only if its determinant is nonnegative. This gives the condition
$$
(1+\theta)^2 \psi_1^2 \psi_2^2 - \frac{1}{4} (\psi_1^2 + \psi_2^2)^2 \geq 0 \Leftrightarrow (1 - \phi_1 \phi_2)^2 - \frac{1}{4}(2 - \phi_1^2 - \phi_2^2)^2 \geq 0,
$$
which is immediately seen to be true, yielding the desired conclusion.

\bigskip

\noindent \underline{Proof of (ii)} A simple computation gives
\begin{align*}
A \partial_t \Theta - \nabla \cdot (A \nabla \Theta) = & \underbrace{ \chi \left[ A \partial_t \theta - \nabla(A \nabla \theta) \right]}_I 
+  \underbrace{(1 - \chi) \left[  A \partial_t |h_1 - h_2|^2 - \nabla \cdot (A \nabla |h_1-h_2|^2 ) \right]}_{II} \\
& + \underbrace{(\Delta \chi - \nabla A \cdot \nabla \chi)(|h_2-h_1|^2 - \theta) - 2 \nabla \chi \cdot \nabla \theta + 4(h_2 - h_1) \nabla \chi \cdot \nabla (h_2-h_1)}_{III}
\end{align*}
We saw in Step 1 that 
\begin{equation}
\label{bumblebee1}
I \leq 0.
\end{equation}
Turning to $II$, it reads
\begin{align*}
II & = 2(1-\chi) A (h_2 - h_1) (\partial_t - \Delta) (h_2 - h_1) - 2 (1-\chi) |\nabla (h_2 - h_1)|^2  + 2(1-\chi)  (h_2 - h_1)  \nabla A \cdot \nabla (h_2 - h_1) \\
& = 2(1-\chi) A (h_2 - h_1) \frac{(d-1)}{2 r^2} [\sin(h_1) - \sin(h_2)] - 2 (1-\chi) |\nabla (h_2 - h_1)|^2  + 2(1-\chi) \nabla A \cdot \nabla (h_2 - h_1) \\
& \lesssim (1-\chi)  |h_2 - h_1|^2 + |h_2 - h_1| |\nabla(h_2 - h_1)|
\end{align*}
where we used in the second equality the equation~\eqref{1} satisfied by $h_1$ and $h_2$, and in the last inequality the fact that $r \gtrsim 1$ on $\operatorname{Supp}(1-\chi)$ as well as the bound $|\nabla A| \lesssim 1$. Using in addition that $|h_2 - h_1|^2 \lesssim \Theta$ and $|h_2 - h_1| |\nabla(h_2 - h_1)| \lesssim \Theta + |\nabla \Theta|$ gives the desired result:
\begin{equation}
\label{bumblebee2}
II \lesssim \Theta + |\nabla \Theta|.
\end{equation}
Finally, using the bounds $|h_2 - h_1|^2 + \theta \lesssim \Theta$ and $|\nabla \theta| + |h_2 - h_1| |\nabla(h_2 - h_1)| \lesssim \Theta + |\nabla \Theta|$ yields once again
\begin{equation}
\label{bumblebee3}
III \lesssim \Theta + |\nabla \Theta|.
\end{equation}
Combining~\eqref{bumblebee1},~\eqref{bumblebee2} and~\eqref{bumblebee3} gives the desired estimate.
\end{proof}

\subsection{Comparison principle for equivariant heat flows}
We state a comparison principle for the equivariant heat flows in $\R^d$.
It is a version of the result \cite[Lemma 4.1]{GGT}  for
the equivariant heat flow in $\R^2$ (see also \cite{CD, GS} for the case when the domain is the disk). 
Since their argument can be adapted for any dimensions, 
we will omit the proof.
A super solution $h(t,r)$ of~\eqref{1} is, by definition, such that
$$
h_t -h_{rr}-\frac{d-1}r h_r+\frac{d-1}{2r^2} \sin (2h) \geq 0
$$
in the sense of distributions. Similarly, a subsolution satisfies the above with a $\leq$ sign.
\begin{lemma}
\label{lemma:comparison}
Let $h_1, h_2 \in BC([0,T] \times[0, \infty))\ \cap \ C^2((0,T) \times (0,\infty))$ 
be respectively sub- and supersolutions of the problem \eqref{1} 
with boundary condition $h_1|_{r=0}=h_2|_{r=0}=0$ and initial data $h_{1,0}$, 
$h_{2,0}$. 
If $h_{1,0}\le h_{2,0}$, then
$$
h_1(t,r)\le h_2(t,r) \qquad {\it for}\quad (t,r) \in [0,\infty) \times [0,T].
$$
\end{lemma}

\subsection{Extension criterion}
In this subsection, we prove the following extension criterion 
for the equivariant heat flow. 
\begin{proposition}
\label{proposition:holder}
Let $h$ be a bounded solution for \eqref{1}
in $[0,T_*)$. If $h$ satisfies 
$$
\left\{ \begin{array}{l}
0\le h(t,r) \lec r^{\beta} \\
|\partial_r h(t,r)| \lesssim \frac{1}{r}
\end{array} \right.
\qquad \textrm{for}\ (t,r) \in 
[0,T_*) \times [0,\infty)
$$
with some $\beta>0$, 
then there exists $\delta>0$ such that 
\begin{align}
\sup_{T_*/2<t<T_*}
\|h(t)\|_{C^\delta([0,\infty))}<\infty.
\label{holder}
\end{align}
In particular, $h$ can be extended beyond $t=T_*$. 
\end{proposition}
\begin{proof}
By the parabolic regularity we have the H\"{o}lder continuity away from the origin, 
and it is enough to show
$$
|h(t,r_2)-h(t,r_1)|\le C|r_1-r_2|^{\delta}
\qquad \textrm{for}\ t \in (T_*/2,T_*), \ 0<r_1<r_2<1.
$$
with some $\delta>0$.
By the assumption, we have
$$
|h(r_2)-h(r_1)| \lec \inf \left\{r_2^\beta,\frac{|r_2-r_1|}{r_1}\right\},
$$
where we abbreviate $t$ variable for simplicity.
Let $\delta>0$ be a positive constant to be chosen later.
Since \eqref{holder} clearly holds if $(\frac{r_2}{2})^{\beta}\le|r_2-r_1|^{\delta}$,
we consider the case $(\frac{r_2}{2})^{\beta}>|r_2-r_1|^{\delta}$.
If we take $0<\delta \le \beta$, 
$$
\frac{r_2}{2}>|r_2-r_1|^{\delta/\beta}\ge|r_2-r_1|,
$$
which implies $r_1 \gec r_2$.
Hence we have
\begin{align*}
|h(r_2)-h(r_1)| 
&\lec \frac{|r_2-r_1|}{r_1}
\\
&\lec  \frac{|r_2-r_1|}{r_2}
\\
&\lec \frac{|r_2-r_1|}{|r_2-r_1|^{\delta/\beta}}
\\
&=|r_2-r_1|^{1-\delta/\beta}.
\end{align*}
Thus choosing $\delta=\frac{\beta}{\beta+1}(\le\beta)$ so that $\delta= 1-\delta/\beta$,
we have the desired estimate \eqref{holder}.
The last assertion follows from the local solvability 
\cite[Theorem 6.1]{KL} for bounded uniformly continuous initial data.
\end{proof}

\noindent
{\bf Acknowledgements:} The authors are grateful to F.H. Lin for very helpful conversations and references.

\end{document}